\documentclass[12pt]{amsart}
\usepackage{latexsym, amsmath, amssymb, amsthm}
\usepackage{epsfig}
\usepackage{amscd}
\usepackage{latexsym}
\usepackage{pstricks,pst-node,cite}

\newtheorem{thm}{Theorem}[section]
\newtheorem{lem}[thm]{Lemma}

\newtheorem{cor}[thm]{Corollary}
\newtheorem{exa}[thm]{Example}

\theoremstyle{remark}

\newcommand{\F}{\mathcal{F}}
\newcommand{\D}{\mathcal{D}}

\newcounter{fignum}
\ifx\blackandwhite\undefined
  \newrgbcolor{darkred}{.9 0 0}\newrgbcolor{emgreen}{0 .9 0}
  \newrgbcolor{pup}{.7  0 .9}\newrgbcolor{xxx}{.8 .8 .8}
  \newrgbcolor{lightgray}{.99 .99 0}\newrgbcolor{darkgray}{.8 .8 .3} \else
\fi
\psset{linecolor=blue}

\psset{linewidth=.7pt,dash=3pt 3pt,doublesep=.05,dotsize=1pt 5}
\SpecialCoor

\begin{document}

\title{Banded surfaces, banded links, band indices and genera of links}

\author{Dongseok Kim}
\address{Department of Mathematics \\Kyonggi University
\\ Suwon, 443-760 Korea}
\email{dongseok@kgu.ac.kr}
\thanks{The first author was supported by Basic Science Research Program through the National Research Foundation of Korea(NRF) funded by the Ministry of Education, Science and Technology(2012R1A1A2006225).}

\author{Young Soo Kwon}
\address{Department of Mathematics \\Yeungnam University \\Kyongsan, 712-749, Korea}
\email{yskwon@yu.ac.kr}

\author{Jaeun Lee}
\address{Department of Mathematics \\Yeungnam University \\Kyongsan, 712-749, Korea}
\email{ysookwon@ynu.ac.kr}

\subjclass[2000]{57M25, 57M27}
\keywords{banded links, Seifert surfaces, banded surfaces, band indices, genera}

\begin{abstract}
Every link is shown to be presentable as a boundary of an unknotted flat banded surface.
A (flat) banded link is defined as a boundary of an unknotted (flat) banded surface.
A link's (flat) band index is defined as the minimum
number of bands required to present the link as boundaries of an unknotted (flat) banded surface.
Banded links of small (flat, respectively) band index are considered here.
Some upper bounds are provided for these band indices of a link using
braid representatives and canonical Seifert surfaces of the link. The relation between the band indices and genera of links
is studied and the band indices of pretzel knots are calculated.
\end{abstract}

\maketitle

\section{Introduction}

A classical goal in knot theory is to find a suitable representative or invariant
for links to establish their classification.
The closure of a braid in the classical Artin group is
a powerful tool~\cite{Birman:Braids, yamada}, though
it fails to achieve the original goal because of the lack
of control on the second Markov move. Representation
theory of quantum groups has allowed extension of links' polynomial
invariants, but still needs further explorations~\cite{Tu94, Witten:pathint}.

Milnor invariants led to a different representative ``the string link" which allowed the classification
of links with small numbers of components up to link homotopy~\cite{Levine, Milnor}.
String links can be defined in different ways~\cite{Milnor, Kauffman, Livingston, LTW}.
The definition employed here is closely related to the banded surfaces on page 191 of Kauffman's \emph{On knots}~\cite{Kauffman}.
Compact orientable surfaces are important in the study of links and $3$-manifolds.
Known as Seifert surface, these surfaces were first proven by Seifert
using an algorithm on a diagram of $L$ that would also take his name~\cite{Seifert:def}.
Some Seifert surfaces feature extra structures,
for example, Seifert surfaces obtained by plumbings annuli have shown
the fibreness of links and surfaces~\cite{Gabai:murasugi1, Gabai:murasugi2, harer:const, Nakamura,
Rudolph:quasipositive2, Rudolph:plumbing, Stallings:const}. Such plumbing surfaces
are well demonstrated~\cite{FHK:openbook, HW:plumbing, Kim:flatbasket}.
However, they have natural geometric restrictions
to avoid ambiguities in construction.
For geometrically flexible Seifert surfaces, \emph{(flat) banded surfaces} are considered here as given by Kauffman~\cite{Kauffman}.
A \emph{(flat) banded link} is defined here as the boundary of a (flat) banded surface.
The flexibility in the definition of banded links facilitates the finding of a banded surface for a given link.
However
for a flat banded surface, analysis of constructions is different.
This article sets out to show the existence of such a flat banded surface for every link.
A link's (flat) band index is defined as the minimum number of bands required to represent it by
a (flat, respectively) banded link. Band indices are then related to the genera of links.

This article is structured as follows. Precise definitions of banded surfaces, flat banded surfaces, banded links and band indices are presented in Section~\ref{def}. Section~\ref{string} demonstrates the existence of banded surfaces and flat banded surfaces. In Section~\ref{smallindex}, some examples of banded links are discussed and
classification theorems for links of small band index are presented.
Upper bounds for band indices are given in Section~\ref{bound}. In Section~\ref{relation}, we
explores the relations between these band indices and the
genera of links which provides some lower bounds for these band indices. The exact band indices for some pretzel knots
are then found.

\section{Preliminaries and definitions} \label{def}

Of the many definitions of string links~\cite{Milnor, Kauffman, Livingston, LTW},
Kauffman's definition of banded surfaces on page 191 of \emph{On knots}~\cite{Kauffman} is used here as a basis.
Let $M$ be a set of $2n$ distinct points on the $x$-axis and an \emph{n-string link}
$S$ is an ordered collection of $n$ disjoint unknotted arcs properly embedded $(\mathbb{R}^3 , M)$ in such a way,
the set of ends of the arcs is exactly $N$. The definition of
the $n$-string link used here is different from Kauffman's, which allows the possibility of knotted arcs~\cite{Kauffman}.
A \emph{framed string link} $(S, \mu)$ is defined with $S$ as an $n$-string link and
$\mu=(m_1$, $m_2$, $\ldots$, $m_n)$ as a framing on $S$. A $3$-string link and
a framed $3$-string link are illustrated in Figure~\ref{stringlink}.

\begin{figure}
$$
\begin{pspicture}[shift=-1.2](-2.2,-1.2)(2.2,2.1)
\pscircle[linewidth=3pt](-2,0){.1}
\pscircle[linewidth=3pt](2,0){.1}
\pscircle[linewidth=3pt](.4,0){.1}
\pscircle[linewidth=3pt](1.2,0){.1}
\pscircle[linewidth=3pt](-.4,0){.1}
\pscircle[linewidth=3pt](-1.2,0){.1}
\psarc[linewidth=1.2pt](-.8,0){1.2}{0}{26}
\psarc[linewidth=1.2pt](-.8,0){1.2}{34}{180}
\psarc[linewidth=1.2pt](1.6,.4){.5656854}{-45}{83}
\psarc[linewidth=1.2pt](1.6,.4){.5656854}{97}{225}
\pccurve[linewidth=1.2pt, angleA=90,angleB=-100](-1.2,0)(-.45,1.1)
\pccurve[linewidth=1.2pt, angleA=70,angleB=180](-.42,1.2)(.4,1.5)
\pccurve[linewidth=1.2pt, angleA=0,angleB=90](.4,1.5)(1.6,.9656854)
\pccurve[linewidth=1.2pt, angleA=-90,angleB=0](1.6,.9656854)(1.1,.4)
\pccurve[linewidth=1.2pt, angleA=180,angleB=-10](.95,.4)(.3,.6)
\pccurve[linewidth=1.2pt, angleA=170,angleB=90](.3,.6)(-.4,0)
\rput(0,-1){{$(a)$}}
\end{pspicture}
\quad\quad
\begin{pspicture}[shift=-1.2](-2.2,-1.2)(2.2,2.1)
\pscircle[linewidth=3pt](-2,0){.1}
\pscircle[linewidth=3pt](2,0){.1}
\pscircle[linewidth=3pt](.4,0){.1}
\pscircle[linewidth=3pt](1.2,0){.1}
\pscircle[linewidth=3pt](-.4,0){.1}
\pscircle[linewidth=3pt](-1.2,0){.1}
\psarc[linewidth=1.2pt](-.8,0){1.2}{0}{26}
\psarc[linewidth=1.2pt](-.8,0){1.2}{34}{180}
\psarc[linewidth=1.2pt](1.6,.4){.5656854}{-45}{83}
\psarc[linewidth=1.2pt](1.6,.4){.5656854}{97}{225}
\pccurve[linewidth=1.2pt, angleA=90,angleB=-100](-1.2,0)(-.45,1.1)
\pccurve[linewidth=1.2pt, angleA=70,angleB=180](-.42,1.2)(.4,1.5)
\pccurve[linewidth=1.2pt, angleA=0,angleB=90](.4,1.5)(1.6,.9656854)
\pccurve[linewidth=1.2pt, angleA=-90,angleB=0](1.6,.9656854)(1.1,.4)
\pccurve[linewidth=1.2pt, angleA=180,angleB=-10](.95,.4)(.3,.6)
\pccurve[linewidth=1.2pt, angleA=170,angleB=90](.3,.6)(-.4,0)
\rput(-2,-.4){{$4$}} \rput(-1.2,-.4){{$-2$}} \rput(1.2,-.4){{$0$}}
\rput(0,-1){{$(b)$}}
\end{pspicture}
$$
\caption{$(a)$ a $3$-string link and $(b)$ a framed $3$-string link.} \label{stringlink}
\end{figure}
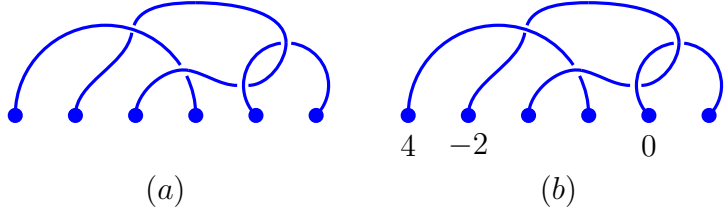

\begin{figure}
$$
\begin{pspicture}[shift=-1.3](-1.4,-1.7)(1.4,2.1)
\pscircle[linewidth=3pt](.4,0){.1}
\pscircle[linewidth=3pt](1.2,0){.1}
\pscircle[linewidth=3pt](-1.2,0){.1}
\pscircle[linewidth=3pt](-.4,0){.1}
\pccurve[linewidth=1.2pt, angleA=60,angleB=-120](-1.2,0)(-.84,.54)
\pccurve[linewidth=1.2pt, angleA=60,angleB=0](-.76,.66)(-.75,1.45)
\pccurve[linewidth=1.2pt, angleA=180,angleB=120](-.85,1.45)(-.8,.6)
\pccurve[linewidth=1.2pt, angleA=-60,angleB=150](-.8,.6)(-.05,.33)
\pccurve[linewidth=1.2pt, angleA=-30,angleB=120](.05,.27)(.4,0)
\pccurve[linewidth=1.2pt, angleA=60,angleB=-150](-.4,0)(0,.3)
\pccurve[linewidth=1.2pt, angleA=30,angleB=-120](0,.3)(.36,.54)
\pccurve[linewidth=1.2pt, angleA=60,angleB=-60](.44,.66)(.42,1.4)
\pccurve[linewidth=1.2pt, angleA=120,angleB=0](.37,1.5)(0,2)
\pccurve[linewidth=1.2pt, angleA=180,angleB=90](0,2)(-.8,1.45)
\pccurve[linewidth=1.2pt, angleA=-90,angleB=170](-.8,1.45)(-.65,1.125)
\pccurve[linewidth=1.2pt, angleA=-10,angleB=-170](-.55,1.1)(.45,1.1)
\pccurve[linewidth=1.2pt, angleA=10,angleB=-90](.57,1.13)(1.1,1.5)
\pccurve[linewidth=1.2pt, angleA=90,angleB=0](1.1,1.5)(.9,1.8)
\pccurve[linewidth=1.2pt, angleA=180,angleB=50](.9,1.8)(.4,1.45)
\pccurve[linewidth=1.2pt, angleA=-130,angleB=70](.4,1.45)(.21,1.13)
\pccurve[linewidth=1.2pt, angleA=-110,angleB=135](.155,1.0)(.4,.6)
\pccurve[linewidth=1.2pt, angleA=-45,angleB=120](.4,.6)(1.2,0)
\rput(0,-1.5){{$(a)$}}
\end{pspicture} \quad\quad
\begin{pspicture}[shift=-1.3](-2.2,-1.7)(2.2,2.1)
\pccurve[doubleline=true, angleA=90,angleB=-120](-1.2,0)(-.7,.8)
\pccurve[doubleline=true, angleA=60,angleB=0](-.7,.8)(-.7,1.45)
\pccurve[doubleline=true, angleA=180,angleB=0](-.7,1.45)(-.9,1.45)
\pccurve[doubleline=true, angleA=180,angleB=120](-.9,1.45)(-.8,.6)
\pccurve[doubleline=true, angleA=-60,angleB=150](-.8,.6)(0,.3)
\pccurve[doubleline=true, angleA=-30,angleB=90](0,.3)(.4,0)
\pccurve[doubleline=true, angleA=90,angleB=-150](-.4,0)(0,.3)
\pccurve[doubleline=true, angleA=30,angleB=-120](0,.3)(.5,.7)
\pccurve[doubleline=true, angleA=60,angleB=-90](.5,.7)(.6,1.3)
\pccurve[doubleline=true, angleA=90,angleB=-60](.6,1.3)(.4,1.8)
\pccurve[doubleline=true, angleA=120,angleB=0](.4,1.8)(0,2)
\pccurve[doubleline=true, angleA=180,angleB=90](0,2)(-.8,1.45)
\pccurve[doubleline=true, angleA=-90,angleB=170](-.8,1.45)(-.6,1.1)
\pccurve[doubleline=true, angleA=-10,angleB=180](-.6,1.1)(0,1.05)
\pccurve[doubleline=true, angleA=0,angleB=-160](0,1.05)(.8,1.2)
\pccurve[doubleline=true, angleA=20,angleB=-90](.8,1.2)(1.1,1.5)
\pccurve[doubleline=true, angleA=90,angleB=0](1.1,1.5)(.9,1.8)
\pccurve[doubleline=true, angleA=180,angleB=50](.9,1.8)(.4,1.45)
\pccurve[doubleline=true, angleA=-130,angleB=70](.4,1.45)(.2,1)
\pccurve[doubleline=true, angleA=-110,angleB=135](.2,1)(.4,.6)
\pccurve[doubleline=true, angleA=-45,angleB=90](.4,.6)(1.2,0)
\pccurve[doubleline=true, angleA=60,angleB=0](-.7,.8)(-.7,1.45)
\pccurve[doubleline=true, angleA=0,angleB=-160](0,1.05)(.8,1.2)
\pccurve[doubleline=true, angleA=60,angleB=-90](.5,.7)(.6,1.3)
\psframe[linecolor=white, fillstyle=solid, fillcolor=white](-2,-1)(2,.1)
\psline(-1.215,.1)(-1.6,.1)(-1.6,-1)(1.6,-1)(1.6,.1)(1.217,.1)
\psline(-1.156,.1)(-.395,.1) \psline(-.317,.1)(.318,.1) \psline(1.148,.1)(.4,.1)
\rput(0,-1.5){{$(b)$}}
\rput(-.5,-.5){{$\mathcal{D}$}}
\rput(.5,-.5){{$+$}}
\end{pspicture}
$$
\caption{$(a)$ A $2$-band and $(b)$ its banded surface with a blackboard framing.} \label{bandsurface}
\end{figure}
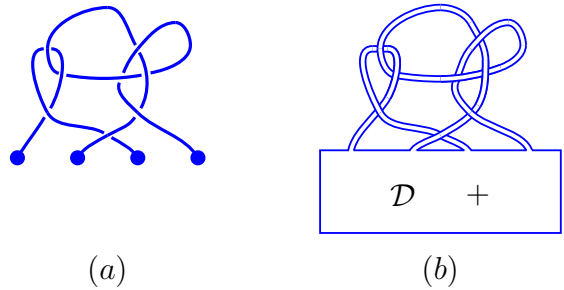

Kauffman obtained a \emph{banded surface} $B$ from an $n$-string link $S$ by attaching a disc $\D$ with a blackboard framing as depicted in Figure~\ref{bandsurface} $(b)$.
Similarly, a framed banded surface is produced here by replacing each arc by a band (each arc in the middle of the band is called a \emph{core}) where a prefixed framing on the band represents $m_i$ full twists; precisely,
$m_i$ is the linking number between a closed path $\alpha$ which is a path sum of the \emph{core} of the $i$-th band of $S$ and any path joining both ends of the core in
$\D$ and its push up $\alpha^{+}$ towards to the positive normal direction
(as indicated ``$+$" in Figure~\ref{bandsurface}).
The linking number discussed here
does not depend on the choice of path on $\D$.
The boundary of an $n$-banded surface has at most $n+1$ components, and the number of components is always congruent to $n+1$ modulo $2$.
In particular, if all framings of an $n$-banded surface
are zero, it is called a \emph{flat $n$-banded surface}.
Further, all bands in these banded surfaces may be linked but not knotted because
they have come from string links. If bands are allowed to be knotted, then \emph{knotted banded surface} results.
Although, banded surfaces are easily constructible for
links, proving the existence of a flat banded surface of a given link is not simple. The existence of
flat banded surfaces from knotted banded surfaces is provided in Theorem~\ref{existencethm}, from closed braids in Theorem~\ref{flatstringthm1} and from canonical Seifert surfaces in Theorem~\ref{flatstringthm3}.
From this, definitions of the band index and the flat band index can follow.
The \emph{band index} of a link $L$, denoted by $B(L)$ is $n$ if there exists an $n$-banded surface $\F$ which is obtained
from an $n$-string link $S$ such that the boundary of the surface $\F$ is $L$ but there does not exist any $k$-banded surface which is obtained from a $k$-string link $T$ for any  positive integer $k$ less than $n$. Similarly the \emph{flat band index} of a link $L$, denoted by $FB(L)$, can be defined as the minimal number of bands
of a flat banded surface whose boundary is $L$.

\section{Existence of flat banded surfaces for a link}\label{string}

This section provides a proof of the existence of a flat banded surface for a given link. Knot theory's standard definitions and notations can be found in~\cite{Adams}.

It is already proven that all links are boundaries of knotted banded surfaces which might
have more than one disc. We first prove that banded surfaces of more than two discs can be modified to banded surfaces of a single disc without changing the link type of the boundary of the banded surfaces in Lemma~\ref{2to1lem}. Next in Lemma~\ref{unknottedlem}, if some bands in a banded surface are knotted,
the surface is shown to be manipulable to make all the bands unknotted without changing the link type of the boundaries of the banded surfaces. Finally Lemma~\ref{flatlem} shows that a non-flat band in a banded surface can be deformed into flat bands without changing the link type of the boundaries of the banded surfaces. These lemmas together can be used to prove in Theorem~\ref{existencethm} that there exists a flat banded surface for a given link.

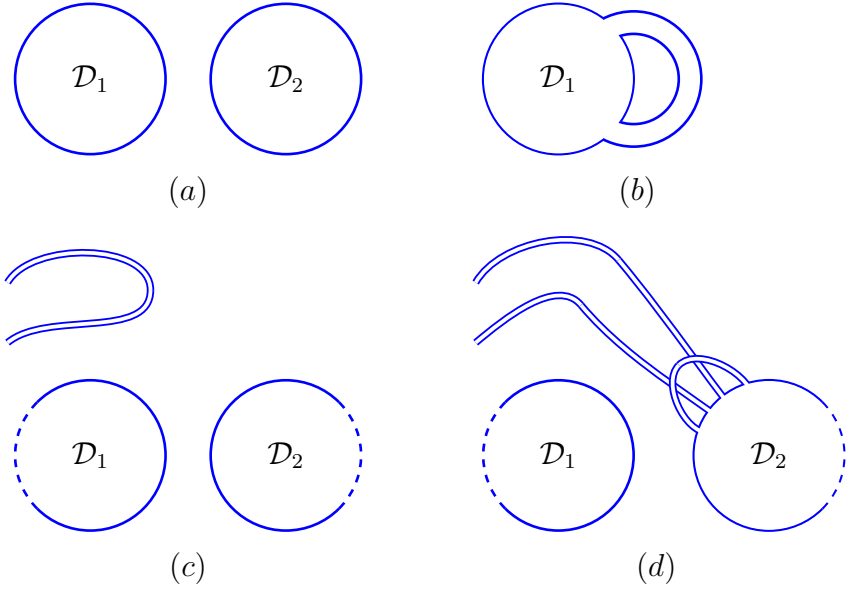
\begin{figure}
$$
\begin{pspicture}[shift=-2](-2.5,-1.7)(2.7,1.2)
\psarc[linewidth=1pt](-1.3,0){1}{0}{360}
\psarc[linewidth=1pt](1.3,0){1}{0}{360}
\rput(0,-1.5){{$(a)$}}
\rput(-1.3,0){{$\mathcal{D}_1$}}
\rput(1.3,0){{$\mathcal{D}_2$}}
\end{pspicture} \quad\quad
\begin{pspicture}[shift=-2](-2.7,-1.7)(2.5,1.2)
\psarc[linewidth=1pt](-1.3,0){1}{-35}{35}
\psarc[linewidth=1pt](-1.3,0){1}{52.5}{307.5}
\psarc[linewidth=1pt](-.3,0){.9}{0}{360}
\psarc[linewidth=1pt](-.3,0){.6}{0}{360}
\pscircle[fillstyle=solid, fillcolor=white, linecolor=white](-1.3,0){.99}
\rput(-.25,-1.5){{$(b)$}}
\rput(-1.3,0){{$\mathcal{D}_1$}}
\end{pspicture}
$$
$$
\begin{pspicture}[shift=-2](-2.5,-2)(2.7,3.1)
\psarc[linewidth=1pt](-1.3,0){1}{-135}{135}
\psarc[linewidth=1pt, linestyle=dashed](-1.3,0){1}{135}{225}
\psarc[linewidth=1pt](1.3,0){1}{45}{315}
\psarc[linewidth=1pt, linestyle=dashed](1.3,0){1}{-45}{45}
\pccurve[doubleline=true, angleA=60,angleB=90](-2.4,2.3)(-.5,2.2)
\pccurve[doubleline=true, angleA=40,angleB=-90](-2.4,1.5)(-.5,2.2)
\rput(0,-1.5){{$(c)$}}
\rput(-1.3,0){{$\mathcal{D}_1$}}
\rput(1.3,0){{$\mathcal{D}_2$}}
\end{pspicture} \quad\quad
\begin{pspicture}[shift=-2](-2.7,-2)(2.5,3.1)
\psarc[linewidth=1pt](-1.3,0){1}{-135}{135}
\psarc[linewidth=1pt, linestyle=dashed](-1.3,0){1}{135}{225}
\psarc[linewidth=1pt](1.5,0){1}{45}{315}
\psarc[linewidth=1pt, linestyle=dashed](1.5,0){1}{-45}{45}
\pccurve[doubleline=true, angleA=60,angleB=130](-2.4,2.3)(-.5,2.6)
\pccurve[doubleline=true, angleA=-50,angleB=130](-.5,2.6)(1.5,0)
\pccurve[doubleline=true, angleA=40,angleB=130](-2.4,1.5)(-1,2)
\pccurve[doubleline=true, angleA=-50,angleB=140](-1,2)(1.5,0)
\pccurve[doubleline=true, angleA=90,angleB=45](1.3,.7)(.3,1.2)
\pccurve[doubleline=true, angleA=-135,angleB=180](.3,1.2)(.7,.3)
\pscircle[fillstyle=solid, fillcolor=white, linecolor=white](1.5,0){.99}
\rput(0,-1.5){{$(d)$}}
\rput(-1.3,0){{$\mathcal{D}_1$}}
\rput(1.5,0){{$\mathcal{D}_2$}}
\end{pspicture}
$$
\caption{$(a), (b)$ Isotopy showing that the trivial link of two components is a banded surface with a single disc and $(c), (d)$ isotopy showing that a banded surface of two disc can be isotop to a banded surface of one disc without changing the link type of the boundary.} \label{twotooneisotopy}
\end{figure}

\begin{lem}
Let $B$ be a knotted banded surface of a link $L$, where $B$ has more than two discs.
Then, there exists a knotted banded surface $C$ of a single disc whose boundary is the link $L$.
\label{2to1lem}
\begin{proof}
A banded surface of two discs need to inductively demonstrated to be isotop to a banded surface of a single disc.
If both discs are not attached by a band, it is a trivial link of two components which can be isotop to a banded surface of one disc and one band as illustrated in~Figure~\ref{twotooneisotopy} $(a), (b)$.
Next consider a disc ${\mathcal{D}_1}$ having at least one band as depicted in Figure~\ref{twotooneisotopy} $(c)$.
This band can be isotoped as depicted in Figure~\ref{twotooneisotopy} $(d)$ by creating a new band to the other disc ${\mathcal{D}_2}$.
Sliding all bands on ${\mathcal{D}_2}$ to ${\mathcal{D}_2}$ along a band between them leads to a banded surface of a single disc by shrinking the disc ${\mathcal{D}_2}$.
\end{proof}
\end{lem}

Before considering the next lemma, we first define a height function $\phi$ of a knotted banded surface as drawn in $\mathbb{R}^2$.
The discs in the knotted banded surface can be assumed to be the unit disc $\{ (x,y)|x^2 + y^2 \le 1 \}$.
The height $\phi(s,t)$ of a point $(s,t)$ in the core of bands in a knotted banded surface can then be defined by the Euclidean distance $d((s,t),(0,0))$. We can perturb the knotted banded surface in general position so that $\phi$ has only finitely many local minima on the core of each band.
By the normal surface theory, if $\phi$ has no local minimum at the core of each band, then it is already a banded surface for which all bands are unknotted.

\begin{lem}
Let $B$ be a knotted banded diagram of a link $L$ where $B$ has $k$-knotted bands.
Then there exists a knotted banded diagram $C$ of the link $L$ which has $(k-1)$-knotted bands.
\label{unknottedlem}
\begin{proof}
If $B$ has a knotted band $b_1$ then
the height function $\phi$ on the core of the band $b_1$ has at least one but finitely many local minima. For such a local minimum, pulling down the band $b_1$ toward the disc leads to the isotopy as
illustrated in Figure~\ref{twotooneisotopy} $(d)$. This process decreases by one the number of the local minima of the height function on the core of the bands in the banded surface $B$ and increase by one the number of band which does not have a
local minimum.
By repeating this process until the height function $\phi$ on the core of the band $b_1$ has no minimum, we obtain a
banded surface $C$ for which the band corresponding to $b_1$ is unknotted. It completes the proof.
\end{proof}
\end{lem}

\begin{figure}
$$
\begin{pspicture}[shift=-1.4](-1.2,-.6)(1.2,2.5)
\psframe[linecolor=lightgray,fillstyle=solid,fillcolor=lightgray](-1,2)(1,2.4)
\psline(-1,2)(-.2,2)(-.2,1.4)
\psline(-.2,.2)(-.2,-.5)
\psline(1,2)(.2,2)(.2,1.4)
\psline(.2,.2)(.2,-.5)
\pccurve[angleA=-90,angleB=90](-.2,.8)(.2,.2)
\pccurve[angleA=-90,angleB=50](.2,.8)(.04,.54)
\pccurve[angleA=-140,angleB=90](-.04,.46)(-.2,.2)
\pccurve[angleA=-90,angleB=90](-.2,1.4)(.2,.8)
\pccurve[angleA=-90,angleB=50](.2,1.4)(.04,1.14)
\pccurve[angleA=-140,angleB=90](-.04,1.06)(-.2,.8)
\end{pspicture} \cong
\begin{pspicture}[shift=-1.4](-2.2,-.6)(2.2,2.5)
\psframe[linecolor=lightgray,fillstyle=solid,fillcolor=lightgray](-2,2)(2,2.4)
\psline(-.2,-.5)(-.2,.62)
\psline(.2,-.5)(.2,.62)
\psline(-2,2)(-1.8,2)
\psline(-1.4,2)(-1,2)
\psline(-.6,2)(-.2,2)(-.2,1.2)
\psline(2,2)(1.8,2)
\psline(1.4,2)(1,2)
\psline(.6,2)(.2,2)(.2,1.2)
\psarc(.4,2){1.4}{180}{0} \psarc(.4,2){1}{180}{0}
\psarc(-.4,2){1.4}{180}{287} \psarc(-.4,2){1}{180}{258}
\psarc(-.4,2){1.4}{315.5}{0} \psarc(-.4,2){1}{293.2}{0}
\end{pspicture}
$$
\caption{The move which removes a full twist in a band by adding two flat bands without changing the  link type.} \label{flatmove}
\end{figure}
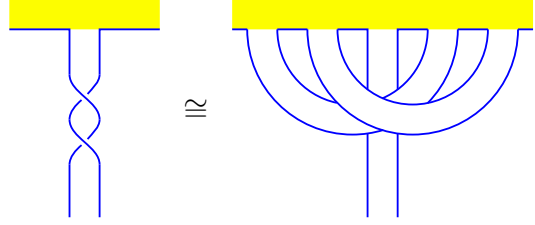

\begin{lem}
Let $B$ be a unknotted banded surface of a link $L$ where $B$ has at least one band with nonzero framing.
Then there exists a flat unknotted banded surface $C$ of the link $L$.
\label{flatlem}
\begin{proof}
If a band in a unknotted banded surface of a link $L$ has $m$ full twists, the move shown in Figure~\ref{flatmove}
can be repeated until a flat band results.
By performing this process for each band, a flat unknotted banded surface $C$ of the link $L$ is left.
\end{proof}
\end{lem}

\begin{thm}
For a given link $L$, there exists a flat unknotted banded surface $F$ whose boundary is the link $L$.
\label{existencethm}
\begin{proof}
For a given link $L$, there exists a knotted banded surface $B$ whose boundary is the link $L$. If $B$ has more than one discs, we apply Lemma~\ref{2to1lem}
to have a knotted banded surface $C$ with a single disc.
If one of bands in the knotted banded surface $C$ is knotted, we apply Lemma~\ref{unknottedlem} inductively to have a banded surface $E$ without a local minimum at the core of each band.
At last, by applying Lemma~\ref{flatlem} we can isotop $E$ to a flat unknotted banded surface $F$ whose boundary is the link $L$.
\end{proof}
\end{thm}

\section{Banded surfaces and flat banded surfaces of small indices}\label{smallindex}

This section considers banded links of band index $0, 1$ or $2$.
One can easily see that a link $L$ has a (flat) band index $0$ if and only if it is the trivial knot.
As any band in a banded surface is not allowed to be knotted, the banded surface of band index $1$ is
a disc with one band of $n$-full twists. Therefore, a link $L$ has band index $1$ if and only if $L$ is a closed
$2$-braid $\overline{(\sigma_1)^{2n}}$ where each string in the braid is oppositely aligned.
A link $L$ has flat band index $1$ if and only if it is the trivial link of two components as depicted in Figure~\ref{twotooneisotopy} $(a)$.
Considered next are links of band index $2$ and flat band index $2$. Lemma~\ref{relationlem} shows the relation between
links of band index $n$ and links of flat band index $n$.

\begin{lem}
A link $L$ is an $n$-banded link if and only if there exists a flat $n$-banded link $F$ such that
$L$ is obtained from $F$ by adding suitable full twists on each band.
\label{relationlem}
\end{lem}

Lemma~\ref{relationlem} allows consideration of only flat $2$-band knots.
The first example, Example~\ref{2flatexample}, can be used to produce many flat $2$-band knots.

\begin{figure}
$$
\begin{pspicture}[shift=-2.4](-3.1,-2.3)(3.1,2.3)
\pccurve[doubleline=true, angleA=135,angleB=180,linewidth=1.2pt](-.2,-1.3)(.5,-.7)
\pccurve[doubleline=true, angleA=0,angleB=-90,linewidth=1.2pt](.5,-.7)(1.5,0)
\pccurve[doubleline=true, angleA=90,angleB=-90,linewidth=1.2pt](1.5,0)(-.5,1.2)
\pccurve[doubleline=true, angleA=-90,angleB=-135,linewidth=1pt](-2.5,0)(0,-1.4)
\pccurve[doubleline=true, angleA=45,angleB=0,linewidth=1.2pt](.2,-1.3)(-.5,-.7)
\pccurve[doubleline=true, angleA=-180,angleB=-90,linewidth=1.2pt](-.5,-.7)(-1.5,0)
\pccurve[doubleline=true, angleA=90,angleB=-90,linewidth=1.2pt](-1.5,0)(.5,1.2)
\pccurve[doubleline=true, angleA=90,angleB=0,linewidth=1.2pt](.5,1.2)(-.5,1.7)
\pccurve[doubleline=true, angleA=180,angleB=90,linewidth=1.2pt](-.5,1.7)(-2.5,0)
\pccurve[doubleline=true, angleA=-90,angleB=-45,linewidth=1.2pt](2.5,0)(0,-1.4)
\pccurve[doubleline=true, angleA=90,angleB=-180,linewidth=1.2pt](-.5,1.2)(.5,1.7)
\pccurve[doubleline=true, angleA=0,angleB=90,linewidth=1.2pt](.5,1.7)(2.5,0)
\pscircle[fillstyle=solid, fillcolor=white, linecolor=white](0,-1.4){.3}
\psarc[linewidth=1pt](0,-1.4){.32}{-26.3}{28}
\psarc[linewidth=1pt](0,-1.4){.32}{53}{127}
\psarc[linewidth=1pt](0,-1.4){.32}{154.7}{205}
\psarc[linewidth=1pt](0,-1.4){.32}{220}{320}
\pccurve[linecolor=darkred, angleA=0,angleB=-90](0,-2)(3,0)
\pccurve[linecolor=darkred, angleA=90,angleB=0](3,0)(0,2)
\pccurve[linecolor=darkred, angleA=180,angleB=90](0,2)(-3,0)
\pccurve[linecolor=darkred, angleA=-90,angleB=180](-3,0)(0,-2)
\pccurve[linecolor=darkred, angleA=0,angleB=-90](0,-.5)(1,0)
\pccurve[linecolor=darkred, angleA=90,angleB=0](1,0)(0,.5)
\pccurve[linecolor=darkred, angleA=180,angleB=90](0,.5)(-1,0)
\pccurve[linecolor=darkred, angleA=-90,angleB=180](-1,0)(0,-.5)
\rput(0,-1.4){{$\mathcal{D}$}}
\end{pspicture}
$$
\caption{A pattern for producing satellite knots of flat band index $2$.} \label{flat2pattern}
\end{figure}
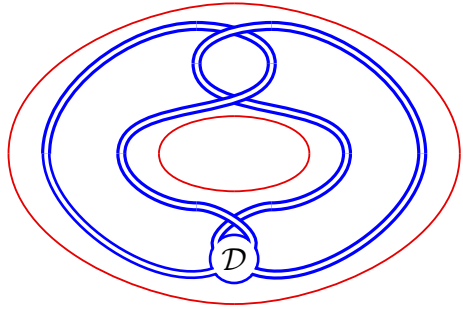

\begin{exa}
Let $B$ be a flat $2$-banded surface, whose cores are embedded in a
thickened annulus, as shown in figure~\ref{flat2pattern}. Assume its conversion into
bands using the blackboard framing convention. If the annulus containing
the bands is used to form a satellite of a knot $K$ then the resulting surface
consists of $2$ bands which are individually unknotted, and so its
boundary, a non-trivial satellite of $K$, is a flat $2$-banded knot. \label{2flatexample}
\end{exa}

This provides examples of flat $2$-banded knots and also demonstrates
a framework for discussing all possible flat $2$-banded knots as stated in Theorem~\ref{2flatthm}.

\begin{figure}
$$
\begin{pspicture}[shift=-2](-2.5,-3.2)(3,2.5)
\psline(-1.4,.15)(1.4,.15)\psline[arrowscale=1.5]{->}(1.2,.15)(1.1,.15)\psline[arrowscale=1.5]{->}(-1.2,.15)(-1.3,.15)
\psline(-1.4,-.15)(1.4,-.15)\psline[arrowscale=1.5]{<-}(1.3,-.15)(1.2,-.15)\psline[arrowscale=1.5]{<-}(-1.1,-.15)(-1.2,-.15)
\psline(-.15,1.4)(-.15,-1.4)\psline[arrowscale=1.5]{->}(-.15,1.2)(-.15,1.1)\psline[arrowscale=1.5]{->}(-.15,-1.2)(-.15,-1.3)
\psline(.15,1.4)(.15,-1.4)\psline[arrowscale=1.5]{->}(.15,1.2)(.15,1.3)\psline[arrowscale=1.5]{->}(.15,-1.2)(.15,-1.1)
\psarc[linewidth=1pt](0,0){1}{-82}{-8}
\psarc[linewidth=1pt](0,0){1}{8}{82}
\psarc[linewidth=1pt](0,0){1}{98}{172}
\psarc[linewidth=1pt](0,0){1}{188}{262}
\pccurve[linestyle=dashed, angleA=-90,angleB=-135](-.15,-1.4)(2,-2)
\pccurve[linestyle=dashed, angleA=45,angleB=-45](2,-2)(2.5,2)
\pccurve[linestyle=dashed, angleA=135,angleB=90](2.5,2)(-.15,1.4)
\pccurve[linecolor=gray, linewidth=1.4pt, angleA=-90,angleB=-135](0,-1.4)(1.9,-1.9)
\pccurve[linecolor=gray, linewidth=1.4pt, angleA=45,angleB=-45](1.9,-1.9)(2.4,1.9)
\pccurve[linecolor=gray, linewidth=1.4pt, angleA=135,angleB=90](2.4,1.9)(0,1.4)
\pccurve[linestyle=dashed, angleA=-90,angleB=-135](.15,-1.4)(1.8,-1.8)
\pccurve[linestyle=dashed, angleA=45,angleB=-45](1.8,-1.8)(2.3,1.8)
\pccurve[linestyle=dashed, angleA=135,angleB=90](2.3,1.8)(.15,1.4)
\pccurve[linestyle=dashed, angleA=180,angleB=-135](-1.4,-.15)(-2,2)
\pccurve[linestyle=dashed, angleA=45,angleB=135](-2,2)(1.5,2)
\pccurve[linestyle=dashed, angleA=-45,angleB=0](1.5,2)(1.4,-.15)
\pccurve[linecolor=gray, linewidth=1.4pt, angleA=180,angleB=-135](-1.4,0)(-1.9,1.9)
\pccurve[linecolor=gray, linewidth=1.4pt, angleA=45,angleB=135](-1.9,1.9)(1.4,1.9)
\pccurve[linecolor=gray, linewidth=1.4pt, angleA=-45,angleB=0](1.4,1.9)(1.4,0)
\pccurve[linestyle=dashed, angleA=180,angleB=-135](-1.4,.15)(-1.8,1.8)
\pccurve[linestyle=dashed, angleA=45,angleB=135](-1.8,1.8)(1.3,1.8)
\pccurve[linestyle=dashed, angleA=-45,angleB=0](1.3,1.8)(1.4,.15)
\pscircle[fillstyle=solid, fillcolor=white, linecolor=white](0,0){.99}
\psline[linecolor=gray, linewidth=1.4pt](0,1.4)(0,-1.4)
\psline[linecolor=gray, linewidth=1.4pt](1.4,0)(-1.4,0)
\rput(0,-3){{$(a)$}}
\rput(.3,.3){{$\mathcal{D}$}}
\end{pspicture} \quad\quad
\begin{pspicture}[shift=-2](-2.5,-3.2)(2.5,2.5)
\psline(-1.4,.15)(1.4,.15)\psline[arrowscale=1.5]{->}(1.2,.15)(1.1,.15)\psline[arrowscale=1.5]{->}(-1.2,.15)(-1.3,.15)
\psline(-1.4,-.15)(1.4,-.15)\psline[arrowscale=1.5]{<-}(1.3,-.15)(1.2,-.15)\psline[arrowscale=1.5]{<-}(-1.1,-.15)(-1.2,-.15)
\psline(-.15,1.4)(-.15,-1.4)\psline[arrowscale=1.5]{->}(-.15,1.2)(-.15,1.1)\psline[arrowscale=1.5]{->}(-.15,-1.2)(-.15,-1.3)
\psline(.15,1.4)(.15,-1.4)\psline[arrowscale=1.5]{->}(.15,1.2)(.15,1.3)\psline[arrowscale=1.5]{->}(.15,-1.2)(.15,-1.1)
\psarc[linewidth=1pt](0,0){1}{-82}{-8}
\psarc[linewidth=1pt](0,0){1}{8}{82}
\psarc[linewidth=1pt](0,0){1}{98}{172}
\psarc[linewidth=1pt](0,0){1}{188}{262}
\pccurve[linestyle=dashed, angleA=-90,angleB=-135](-.15,-1.4)(2,-2)
\pccurve[linestyle=dashed, angleA=45,angleB=0](2,-2)(1.4,.15)
\pccurve[linestyle=dashed, angleA=-90,angleB=-135](.15,-1.4)(1.8,-1.8)
\pccurve[linestyle=dashed, angleA=45,angleB=0](1.8,-1.8)(1.4,-.15)
\pccurve[linestyle=dashed, angleA=90,angleB=45](.15,1.4)(-2,2)
\pccurve[linestyle=dashed, angleA=-135,angleB=180](-2,2)(-1.4,-.15)
\pccurve[linestyle=dashed, angleA=90,angleB=45](-.15,1.4)(-1.8,1.8)
\pccurve[linestyle=dashed, angleA=-135,angleB=180](-1.8,1.8)(-1.4,.15)
\pscircle[fillstyle=solid, fillcolor=white, linecolor=white](0,0){.99}
\rput(0,-3){{$(b)$}}
\rput(0,0){{$\mathcal{D}$}}
\end{pspicture}
$$
\caption{The local structure around the disc $\mathcal{D}$ of a flat $2$-banded surface $B$; the dark gray line represents
an unknotted $2$ component link $L$.} \label{localfig}
\end{figure}
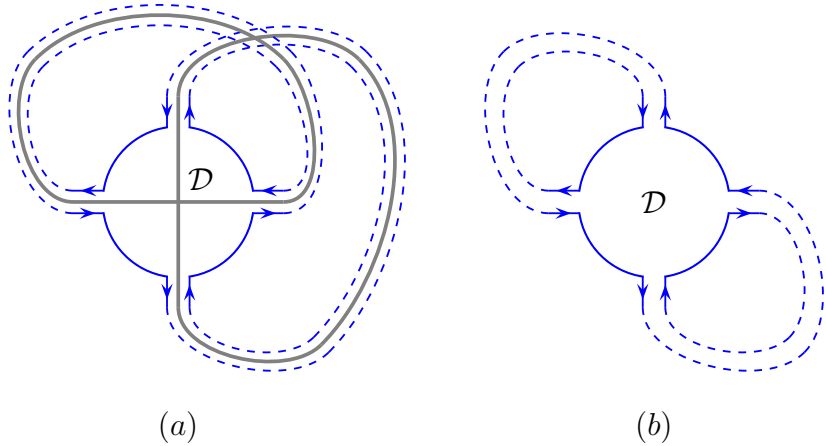

\begin{thm}
A knot $K$ is a flat $2$-banded knot if and only if there exits a $2$ component link $L$ whose
components $L_1, L_2$ are unknotted such that $K$ is the boundary of a surface which is obtained from
two untwisted bands that are deformation retracts of $L$ by plumbing at one of the crossing between $L_1$ and $L_2$.
\label{2flatthm}
\begin{proof}
Let $K$ be a flat $2$-banded knot with a
flat banded surface $B$. The local structure around the disc $\mathcal{D}$ must be as shown in Figure~\ref{localfig} $(a)$
because the boundary shown in Figure~\ref{localfig} $(b)$ has three components.
Consider the core of the bands in $B$, as indicted by the dark gray line in Figure~\ref{localfig} $(a)$.
It is an unknotted $2$ component link $L$ whose
components are unknotted such that $K$ is the boundary of a surface obtained from
two untwisted bands which is a deformation retract of $L$ by plumbing at one of the crossings of $L$.

Conversely, let $L$ be a $2$ component link with unknotted components.
$L$ can be represented with an untwisted band around each component. The two bands can then be
plumbed together at one of the crossings. The result is a
surface with two unknotted flat bands, whose boundary gives a flat $2$-banded
knot so long as it is not the trivial knot. Therefore, the knot $K$ is a flat $2$-banded knot.
\end{proof}
\end{thm}

The construction method in Theorem~\ref{2flatthm} can be used to establish a similar classification theorem for
flat $n$-banded knots.

Example~\ref{2flat1genus} and~\ref{2bandnot2flat} respectively consider knots of genus $1$ which are not $2$-banded knots
and a $2$-banded knots which is not flat $2$-banded knots.

\begin{exa}
Let $K$ be a non-trivial knot. Let $L$ be the Whitehead double of $K$.
$L$ has genus $1$ and band index is strictly greater
than $2$.
\label{2flat1genus}
\begin{proof}
$L$ has genus $1$ and is spanned by a $2$-band surface $F$ lying inside a
solid torus neighborhood of $K$. The only unknotted curves in such a solid
torus must lie inside a ball in the solid torus, assuming that $K$ is non-trivial.
It is not possible to find two independent unknotted curves spanning the
homology $H_1 (F)$, since they would both have to lie in a ball in the solid
torus, and hence both represent $0$ in the homology group. However $H_1(F)$
maps onto the homology of the solid torus under the inclusion map. Therefore the band index of any Whitehead double is strictly greater than $2$.
\end{proof}
\end{exa}

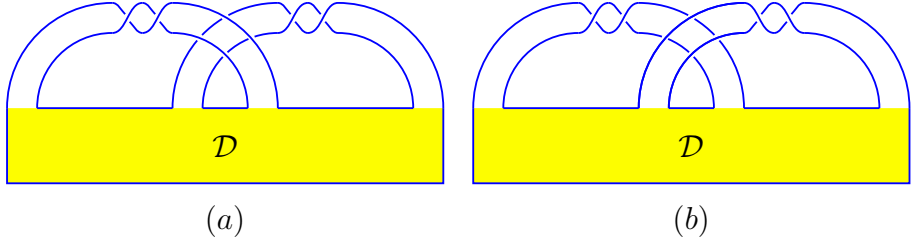
\begin{figure}
$$
\begin{pspicture}[shift=-1.4](-2.4,-1.7)(3.8,2)
\psframe[linecolor=lightgray,fillstyle=solid,fillcolor=lightgray](-2.2,-1)(3.6,0)
\psline(-2.2,0)(-2.2,-1)(3.6,-1)(3.6,0)
\psline(3.2,0)(1.4,0)
\psline(1,0)(.4,0)
\psline(0,0)(-1.8,0)
\pccurve[angleA=0,angleB=180](-.8,1)(-.4,1.4)
\pccurve[angleA=0,angleB=180](-.4,1)(0,1.4)
\pccurve[angleA=0,angleB=135](-.8,1.4)(-.63,1.25)
\pccurve[angleA=-45,angleB=180](-.57,1.15)(-.4,1)
\pccurve[angleA=0,angleB=135](-.4,1.4)(-.23,1.25)
\pccurve[angleA=-45,angleB=180](-.17,1.15)(0,1)
\pccurve[angleA=0,angleB=180](1.4,1)(1.8,1.4)
\pccurve[angleA=0,angleB=180](1.8,1)(2.2,1.4)
\pccurve[angleA=0,angleB=135](1.4,1.4)(1.57,1.25)
\pccurve[angleA=-45,angleB=180](1.63,1.15)(1.8,1)
\pccurve[angleA=0,angleB=135](1.8,1.4)(1.97,1.25)
\pccurve[angleA=-45,angleB=180](2.03,1.15)(2.2,1)
\psarc(0,0){1.4}{0}{90} \psarc(0,0){1}{0}{90}
\psarc(-.8,0){1.4}{90}{180} \psarc(-.8,0){1}{90}{180}
\psarc(2.2,0){1.4}{0}{90} \psarc(2.2,0){1}{0}{90}
\psarc(1.4,0){1.4}{90}{118} \psarc(1.4,0){1.4}{122}{136} \psarc(1.4,0){1.4}{140}{180}
\psarc(1.4,0){1}{90}{108} \psarc(1.4,0){1}{113}{132} \psarc(1.4,0){1}{137}{180}
\rput(.7,-.5){{$\mathcal{D}$}}
\rput(.7,-1.5){{$(a)$}}
\end{pspicture}
\begin{pspicture}[shift=-1.4](-2.4,-1.7)(3.8,2)
\psframe[linecolor=lightgray,fillstyle=solid,fillcolor=lightgray](-2.2,-1)(3.6,0)
\psline(-2.2,0)(-2.2,-1)(3.6,-1)(3.6,0)
\psline(3.2,0)(1.4,0)
\psline(1,0)(.4,0)
\psline(0,0)(-1.8,0)
\pccurve[angleA=0,angleB=180](-.8,1)(-.4,1.4)
\pccurve[angleA=0,angleB=180](-.4,1)(0,1.4)
\pccurve[angleA=0,angleB=135](-.8,1.4)(-.63,1.25)
\pccurve[angleA=-45,angleB=180](-.57,1.15)(-.4,1)
\pccurve[angleA=0,angleB=135](-.4,1.4)(-.23,1.25)
\pccurve[angleA=-45,angleB=180](-.17,1.15)(0,1)
\pccurve[angleA=0,angleB=180](1.4,1.4)(1.8,1)
\pccurve[angleA=0,angleB=180](1.8,1.4)(2.2,1)
\pccurve[angleA=0,angleB=-135](1.4,1)(1.57,1.15)
\pccurve[angleA=45,angleB=180](1.63,1.25)(1.8,1.4)
\pccurve[angleA=0,angleB=-135](1.8,1)(1.97,1.15)
\pccurve[angleA=45,angleB=180](2.03,1.25)(2.2,1.4)
\psarc(0,0){1.4}{0}{40} \psarc(0,0){1.4}{44}{58} \psarc(0,0){1.4}{62}{90}
\psarc(0,0){1}{0}{43} \psarc(0,0){1}{48}{67} \psarc(0,0){1}{72}{90}
\psarc(-.8,0){1.4}{90}{180} \psarc(-.8,0){1}{90}{180}
\psarc(2.2,0){1.4}{0}{90} \psarc(2.2,0){1}{0}{90}
\psarc(1.4,0){1.4}{90}{180} \psarc(1.4,0){1}{90}{180}
\psarc(1.4,0){1.4}{90}{180} \psarc(1.4,0){1}{90}{180}
\rput(.7,-.5){{$\mathcal{D}$}}
\rput(.7,-1.5){{$(b)$}}
\end{pspicture}
$$
\caption{$2$-banded surfaces of $(a)$ the trefoil knot and $(b)$ the figure eight knot.} \label{2bandsfig}
\end{figure}

\begin{exa}
Let $K_1$ be the trefoil and $K_2$ be the figure eight knot.
Then both have band index $2$ but neither have flat band index $2$.
\label{2bandnot2flat}
\begin{proof}
Figure~\ref{2bandsfig} shows two $2$-banded surfaces whose boundaries are the
trefoil knot and the figure eight knot. To show that these two knots are not flat $2$-banded knot,
we consider Conway polynomials of flat $2$-banded knots.
Any knot with flat band index $2$ has a Seifert matrix $A$ of the form
$$A = \left[\begin{matrix} 0 & k \\
k + 1 & 0 \end{matrix} \right]$$
based on the flat bands where $k$ is an integer.
Its Conway polynomial is then $1-k(k+1)z^2$.
Neither the trefoil nor the
figure eight knot, which both have band index $2$, has a Conway polynomial
of this form because the Conway polynomials of the
trefoil knot and the figure eight knot are $1+z^2$ and $1-z^2$.
\end{proof}
\end{exa}


\begin{figure}
$$
\begin{pspicture}[shift=-2.97](-1.1,-2.3)(1.8,2.3)
\pccurve[doubleline=true, angleA=90,angleB=180,linewidth=1pt](.3,1.5)(.45,1.9)
\pccurve[doubleline=true, angleA=0,angleB=90,linewidth=1pt](.45,1.9)(.6,1.5)
\pccurve[doubleline=true, angleA=90,angleB=180,linewidth=1pt](.9,1.5)(1.05,1.9)
\pccurve[doubleline=true, angleA=0,angleB=90,linewidth=1pt](1.05,1.9)(1.2,1.5)
\pccurve[doubleline=true, angleA=-90,angleB=180,linewidth=1pt](.3,0)(.45,-.4)
\pccurve[doubleline=true, angleA=0,angleB=-90,linewidth=1pt](.45,-.4)(.6,0)
\pccurve[doubleline=true, angleA=-90,angleB=180,linewidth=1pt](.9,0)(1.05,-.4)
\pccurve[doubleline=true, angleA=0,angleB=-90,linewidth=1pt](1.05,-.4)(1.2,0)
\pspolygon[fillcolor=lightgray, fillstyle=solid](0,0)(1.5,0)(1.5,1.5)(0,1.5)(0,0)
\rput(.75,.75){{$B$}}
\rput(-.8,.75){{$K$}}
\rput(-.4,.75){{$=$}}
\end{pspicture} \quad = \quad
\begin{pspicture}[shift=-2.97](-.2,-2.3)(3.1,2.3)
\pccurve[doubleline=true, angleA=90,angleB=180,linewidth=1pt](.3,1.5)(.45,1.9)
\pccurve[doubleline=true, angleA=0,angleB=90,linewidth=1pt](.45,1.9)(.6,1.5)
\pccurve[doubleline=true, angleA=90,angleB=180,linewidth=1pt](.9,1.5)(1.05,1.9)
\pccurve[doubleline=true, angleA=0,angleB=90,linewidth=1pt](1.05,1.9)(1.2,1.5)
\pccurve[doubleline=true, angleA=-90,angleB=90,linewidth=1pt](.3,0)(.3,-1)
\pccurve[doubleline=true, angleA=-90,angleB=90,linewidth=1pt](.6,0)(.6,-1)
\pccurve[doubleline=true, angleA=-90,angleB=90,linewidth=1pt](.9,0)(1.2,-1)
\pccurve[doubleline=true, angleA=-90,angleB=90,linewidth=1pt](1.2,0)(1.5,-1)
\pccurve[doubleline=true, angleA=90,angleB=180,linewidth=1pt](.9,-1)(1.35,-.5)
\pccurve[doubleline=true, angleA=0,angleB=90,linewidth=1pt](1.35,-.5)(1.8,-1)
\pspolygon[fillcolor=lightgray, fillstyle=solid](0,0)(1.5,0)(1.5,1.5)(0,1.5)(0,0)
\pspolygon[fillcolor=lightgray, fillstyle=solid, linecolor=lightgray](0,-1)(2.1,-1)(2.1,-1.8)(0,-1.8)(0,-1)
\psline(2.1,-1)(2.1,-1.8)(0,-1.8)(0,-1)(.272,-1)
\psline(.33,-1)(.572,-1) \psline(.63,-1)(.872,-1) \psline(.93,-1)(1.172,-1) \psline(1.23,-1)(1.472,-1)
\psline(1.53,-1)(1.772,-1) \psline(1.83,-1)(2.11,-1)
\rput(.75,.75){{$B$}}
\rput(1.05,-1.4){{$\mathcal{D}$}}
\end{pspicture}
$$
\caption{The reverse parallel of any two-bridge knot K presented by using
three unknotted bands.} \label{twobridge}
\end{figure}
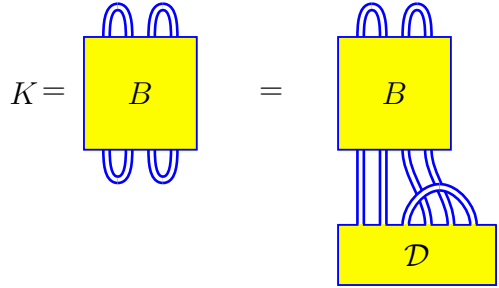

The exact flat band index of the trefoil knot and figure eight knot are given in Example~\ref{flattrefoil}
and~\ref{flatfigureeight}.
As an explicit application, the reverse parallel of a $k$-bridge knot should
have a $(2k-1)$-banded presentation for some framing of the knot. This same
method yields a presentation of the Whitehead double of the trefoil as a flat
$4$-banded knot. This method is illustrated in Example~\ref{bridgeexa}


\begin{exa} \label{bridgeexa}
Any two-bridge knot $K$ can be presented as the plat closure of some $4$-
string braid $B$. A reverse parallel of $K$ can then be given by the boundary of a
band along $K$. This technique allow the boundary of
a banded surface with $3$ unknotted bands to be presented as in Figure~\ref{twobridge}.
\end{exa}

\section{Upper bounds for band index and flat band index}\label{bound}

The existence of a banded surface and a flat banded surface of a given link $L$ is proved
by Theorem~\ref{existencethm}. However, it was shown from (knotted) banded diagrams but we could not locate the complete list of links' banded diagrams.
To find a banded surface and a flat banded surface of a given link $L$,
we consider more popular representatives of links:
braid representatives and canonical Seifert surfaces of the link.
Consequently upper bounds for band index and flat band index can be derived from
braid representations and canonical Seifert surfaces using the induced graphs.

\subsection{Banded surfaces and band indices}

An upper bound for the band index of $L$ can be obtained from its canonical Seifert surface as follow.
First some terminology is defined.
A canonical Seifert surface is a Seifert surface $\F$ of a link $L$ obtained by applying Seifert's
algorithm to a link diagram $D(L)$ as depicted in Figure~\ref{stringgraph}.
From such a canonical Seifert surface, a (signed) graph $G(L)$ can be constructed by collapsing discs to vertices
and half twisted bands to signed edges as illustrated in the right
side of Figure~\ref{stringgraph}. This graph is called the \emph{induced graph}
of the canonical Seifert surface of the link $L$. The number of Seifert
circles (half twisted bands), denoted by $s(\F)(c(\F))$, is the
cardinality of the vertex set (edge set, respectively) of the induced graph $G(L)$.
If the Seifert surface $\F$ is connected, its induced graph $G(L)$
is also connected. For terms in graph theory, we refer the readers
to~\cite{GT1}. The number of edges of the
spanning tree of a connected graph with $n$ vertices is $n-1$.
A theorem giving the
upper bound of the band index of $L$ can be derived from its canonical Seifert surface.

\begin{figure}
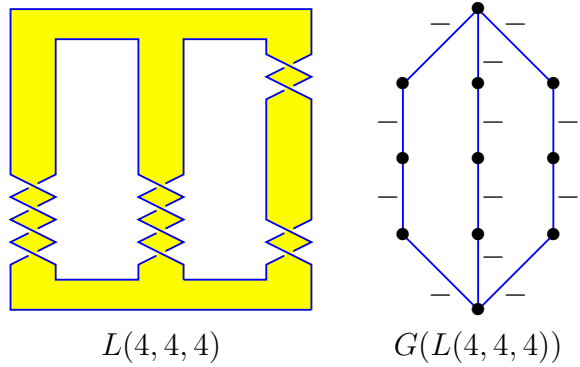

$$
\pspicture[shift=-2.1](-2.1,-2.7)(2.4,2.1)
\pspolygon[linecolor=lightgray,fillstyle=solid,fillcolor=lightgray](-2,-2)(-2,-1.4)
(-1.7,-1.25)(-1.4,-1.4)(-1.4,-1.6)(-.3,-1.6)(-.3,-1.4)(0,-1.25)(.3,-1.4)(.3,-1.6)
(1.4,-1.6)(1.4,-1.4)(1.7,-1.25)(2,-1.4)(2,-2)(-2,-2)
\pspolygon[linecolor=lightgray,fillstyle=solid,fillcolor=lightgray](-2,2)(-2,-.2)
(-1.7,-.35)(-1.4,-.2)(-1.4,1.6)(-.3,1.6)(-.3,-.2)(0,-.35)(.3,-.2)(.3,1.6)
(1.4,1.6)(1.4,1.4)(1.7,1.25)(2,1.4)(2,2)(-2,2)
\pspolygon[linecolor=lightgray,fillstyle=solid,fillcolor=lightgray](-2,-1.1)(-1.7,-1.25)(-1.4,-1.1)(-1.7,-.95)(-2,-1.1)
\pspolygon[linecolor=lightgray,fillstyle=solid,fillcolor=lightgray](-2,-.8)(-1.7,-.95)(-1.4,-.8)(-1.7,-.65)(-2,-.8)
\pspolygon[linecolor=lightgray,fillstyle=solid,fillcolor=lightgray](-2,-.5)(-1.7,-.65)(-1.4,-.5)(-1.7,-.35)(-2,-.5)
\pspolygon[linecolor=lightgray,fillstyle=solid,fillcolor=lightgray](-.3,-1.1)(0,-1.25)(.3,-1.1)(0,-.95)(-.3,-1.1)
\pspolygon[linecolor=lightgray,fillstyle=solid,fillcolor=lightgray](-.3,-.8)(0,-.95)(.3,-.8)(0,-.65)(-.3,-.8)
\pspolygon[linecolor=lightgray,fillstyle=solid,fillcolor=lightgray](-.3,-.5)(0,-.65)(.3,-.5)(0,-.35)(-.3,-.5)
\pspolygon[linecolor=lightgray,fillstyle=solid,fillcolor=lightgray](2,-1.1)(1.7,-1.25)(1.4,-1.1)(1.7,-.95)(2,-1.1)
\pspolygon[linecolor=lightgray,fillstyle=solid,fillcolor=lightgray](2,1.1)(1.7,1.25)(1.4,1.1)(1.7,.95)(2,1.1)
\pspolygon[linecolor=lightgray,fillstyle=solid,fillcolor=lightgray](2,-.8)(1.7,-.95)(1.4,-.8)(1.4,.8)(1.7,.95)(2,.8)(2,-.8)
\psline(2,-2)(-2,-2)(-2,-1.6)(-2,-1.4)(-1.76,-1.28)
\psline(-1.64,-1.22)(-1.4,-1.1)(-2,-.8)(-1.76,-.68)
\psline(-1.64,-.92)(-1.4,-.8)(-2,-.5)(-1.76,-.38)
\psline(-1.64,-.62)(-1.4,-.5)(-2,-.2)(-2,2)(2,2)(2,1.4)(1.76,1.28)
\psline(-1.64,-.32)(-1.4,-.2)(-1.4,1.6)(-.3,1.6)(-.3,-.2)(.3,-.5)(.06,-.62)
\psline(-1.76,-.98)(-2,-1.1)(-1.4,-1.4)(-1.4,-1.6)(-.3,-1.6)(-.3,-1.4)(-.06,-1.28)
\psline(.3,-1.6)(.3,-1.4)(-.3,-1.1)(-.06,-.98)
\psline(.06,-1.22)(.3,-1.1)(-.3,-.8)(-.06,-.68)
\psline(.06,-.92)(.3,-.8)(-.3,-.5)(-.06,-.38)
\psline(.06,-.32)(.3,-.2)(.3,1.6)(1.4,1.6)(1.4,1.4)(2,1.1)(1.76,.98)
\psline(1.64,1.22)(1.4,1.1)(2,.8)(2,-.8)(1.76,-.92)
\psline(1.64,.92)(1.4,.8)(1.4,-.8)(2,-1.1)(1.76,-1.22)
\psline(1.64,-.98)(1.4,-1.1)(2,-1.4)(2,-2)
\psline(1.64,-1.28)(1.4,-1.4)(1.4,-1.6)(.3,-1.6)
\rput(0,-2.5){\rnode{a}{$L(4,4,4)$}}
\endpspicture
\quad
\pspicture[shift=-2.1](-1.4,-2.7)(1,2.1)
\psline(0,2)(0,-2)(1,-1)(1,1)(0,2)(-1,1)(-1,-1)(0,-2)
\rput(0,0){$\bullet$}
\rput(0,1){$\bullet$}
\rput(0,2){$\bullet$}
\rput(0,-1){$\bullet$}
\rput(0,-2){$\bullet$}
\rput(1,1){$\bullet$}
\rput(1,0){$\bullet$}
\rput(1,-1){$\bullet$}
\rput(-1,1){$\bullet$}
\rput(-1,0){$\bullet$}
\rput(-1,-1){$\bullet$}
\rput(.5,1.8){\rnode{a}{$-$}}
\rput(.5,-1.8){\rnode{a}{$-$}}
\rput(-.5,1.8){\rnode{a}{$-$}}
\rput(-.5,-1.8){\rnode{a}{$-$}}
\rput(1.2,.5){\rnode{a}{$-$}}
\rput(1.2,-.5){\rnode{a}{$-$}}
\rput(-1.2,.5){\rnode{a}{$-$}}
\rput(-1.2,-.5){\rnode{a}{$-$}}
\rput(.2,-.5){\rnode{a}{$-$}}
\rput(.2,.5){\rnode{a}{$-$}}
\rput(.2,1.3){\rnode{a}{$-$}}
\rput(.2,-1.3){\rnode{a}{$-$}}
\rput(0,-2.5){\rnode{a}{$G(L(4,4,4))$}}
\endpspicture
$$
\caption{A canonical Seifert surface of a pretzel link $L(4,4,4)$ and its corresponding induced graph $G(L(4,4,4))$.} \label{stringgraph}
\end{figure}

\begin{thm}
Let $\F$ be a canonical Seifert surface of a link diagram $S$ of $L$
with $s(\F)$ Seifert circles and $c(\F)$ half twisted bands. Let $T$
be a spanning tree of its induced graph $G(L)$ which has exactly
$s(\F)-1$ edges. Then there exists a banded surface $B$ of $c(\F)-s(\F)+1$ bands whose boundary is $L$, $i.e.$, the band index of $L$, is less than or equal to
$c(\F)-s(\F)+1.$ \label{stringthm2}
\begin{proof}
Let $\D$ be the disc corresponding to the spanning tree $T$. For any half
twisted band $t(e)$ with a sign $\epsilon(e)$
which is not a part of the disc $\D$, $i. e.,$ an
edge $e$ in $E(G)-T$ in its induced graph $G(L)$ a unique path $W$ can be chosen in the spanning tree $T$ which joins the end vertices of the edge $e$.
Let $k$ be the sign sum of edges in the path $W$. Then $k+\epsilon(e)$
has to be an even integer because $\F$ is an oriented surface.
Let $n_e=\dfrac{k+\epsilon(e)}{2}$. By sliding the half twisted band $t(e)$ which represents
the generators of the homology of the surface, a banded surface $B$ can be obtained
and $n_e$ is the number of full twists in the band presented by $t(e)$.
Since there are $c(\F)-s(\F)+1$ edges in $E(G)-T$, its band index is less than or equal
to $c(\F)-s(\F)+1$.
\end{proof}
\end{thm}

Theorem~\ref{stringthm2} is applied in the calculation of the band index of the pretzel link $L(4,4,4)$ in the following example.

\begin{exa}
The band index of the pretzel link $L(4,4,4)$ depicted in Figure~\ref{stringgraph} is $2$. \label{examp}
\begin{proof}
The shaded region in the figure is a Seifert surface of the link $L(4,4,4)$.
The figure shows that $s(\F)=11, c(\F)=12$. By applying
Theorem~\ref{stringthm2}, the band index of $L$ is less than or equal to
$c(\F)-s(\F)+1 =12-11+1=2$. Recall that links of the band index $1$ have two components. Since $L(4,4,4)$ has $3$ components, it is not a link of the band index $1$ and its band index must be $2$.
\end{proof}
\end{exa}

Example~\ref{examp} demonstrates
the sharp inequality in Theorem~\ref{stringthm2}. One can easily observe that Example~\ref{examp} can be generalized for
all pretzel links of form $L(2p, 2q, 2r)$ that their band index is $2$ by a similar proof.
Later, Theorem~\ref{genusthm1} determines the band index of all $n$-pretzel knots.

A banded surface can be found from the closure of
a braid $\beta$ in a classical Artin group $B_n$. By applying the first Markov move, the conjugation in $B_n$,
the braid $\beta$ can be assumed to be of the form
$\beta = (\sigma_{1}$$\sigma_{2}$$
\ldots $$\sigma_{n-2}$ $\sigma_{n-1})$ $W$.
This braid representative can be used to find the
following upper bound for the band index of the closed braid $\overline{\beta}$ for which all the framings of bands are either $0$ or $1$.

\begin{thm}
Let $L$ be a closed n-braid $\overline{\beta}$
where the braid $\beta$ can be written as $\sigma_{1}$$\sigma_{2}$$
\ldots $$\sigma_{n-2}$ $\sigma_{n-1} W$ and the length of the word $W$ is $m$.
Then the band index of $L$ is less than or equal to $m$, $i. e.$
$B(L)\le m$. \label{stringthm1}
\end{thm}
\begin{proof}
First select a disc $\D$ which is obtained from $n$ disjoint disks
by attaching $(n-1)$ twisted bands represented by
$\sigma_{1}$$\sigma_{2}$$
\ldots $$\sigma_{n-2}$ $\sigma_{n-1}$. For each letter
in the word $W$, attach a half twisted band to $\D$ which will represent a band
in the desired banded surface.
Counting the linking number shows that each half twisted band corresponding to a negative letter,
$\sigma_i^{-1}$ has framing $0$ and
each half twisted band corresponding to a positive letter, $\sigma_i$ has a framing $1$.
\end{proof}

The figure eight knot demonstrates the sharp inequality in Theorem~\ref{stringthm1} as follow.
The figure eight knot is a closed braid of $\sigma_1^-1 \sigma_2 \sigma_1^-1 \sigma_2$.
In Theorem~\ref{stringthm1}, we may choose $\beta$ of a form  $\sigma_{1}^{-1}$$\sigma_{2}W$ and we obtain
that the band index of the figure eight knot is less than or equal to $2$. Example~\ref{2bandnot2flat} showed that the band index of the figure eight knot is $2$.

\subsection{Flat banded surfaces and flat band indices}

The key ingredients to construct a flat band surface for a given link are Lemma~\ref{flatstringlem1} and
the method to change the sign of a twisted band by adding two flat annuli, illustrated in Figure~\ref{3annuli}, which was
first reported by R. Furihata, M. Hirasawa and T. Kobayashi~\cite{FHK:openbook}.

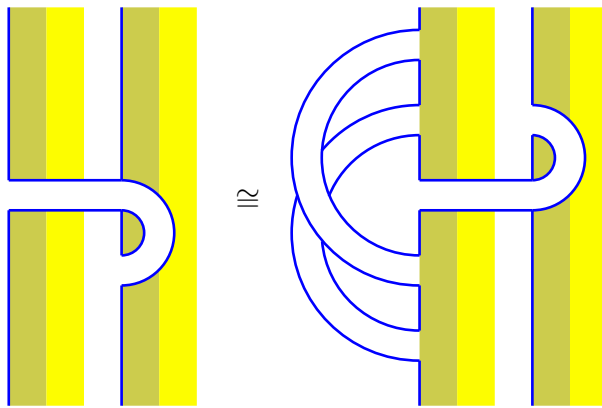
\begin{figure}
$$
\begin{pspicture}[shift=-2.8](-1,-3)(1.9,2.8)
\psframe[linecolor=darkgray,fillstyle=solid,fillcolor=darkgray](1,2.5)(.5,-2.8)
\psframe[linecolor=lightgray,fillstyle=solid,fillcolor=lightgray](1,-2.8)(1.5,2.5)
\pscircle[linecolor=white,fillstyle=solid,fillcolor=white](.5,-.5){.7}
\pscircle[linecolor=darkgray,fillstyle=solid,fillcolor=darkgray](.5,-.5){.3}
\psframe[linecolor=white,fillstyle=solid,fillcolor=white](-.3,-2.8)(.5,2.5)
\psframe[linecolor=darkgray,fillstyle=solid,fillcolor=darkgray](-1,.2)(-.5,2.5)
\psframe[linecolor=darkgray,fillstyle=solid,fillcolor=darkgray](-1,-.2)(-.5,-2.8)
\psframe[linecolor=lightgray,fillstyle=solid,fillcolor=lightgray](0,.2)(-.5,2.5)
\psframe[linecolor=lightgray,fillstyle=solid,fillcolor=lightgray](0,-.2)(-.5,-2.8)
\psline[linewidth=1pt](.5,2.5)(.5,.2)
\psline[linewidth=1pt](.5,-.8)(.5,-.2)
\psline[linewidth=1pt](.5,-1.2)(.5,-2.8)
\psline[linewidth=1pt](-1,2.5)(-1,.2)
\psline[linewidth=1pt](-1,-.2)(-1,-2.8)
\psarc[linewidth=1pt](.5,-.5){.3}{-90}{90}
\psarc[linewidth=1pt](.5,-.5){.7}{-90}{90}
\psline[linewidth=1pt](-1,.2)(.5,.2)
\psline[linewidth=1pt](-1,-.2)(.5,-.2)
\end{pspicture} \cong
\begin{pspicture}[shift=-2.8](-3,-3)(1.5,2.8)
\psframe[linecolor=darkgray,fillstyle=solid,fillcolor=darkgray](1,2.5)(.5,-2.8)
\psframe[linecolor=lightgray,fillstyle=solid,fillcolor=lightgray](1,-2.8)(1.5,2.5)
\pscircle[linecolor=white,fillstyle=solid,fillcolor=white](.5,.5){.7}
\pscircle[linecolor=darkgray,fillstyle=solid,fillcolor=darkgray](.5,.5){.3}
\psframe[linecolor=white,fillstyle=solid,fillcolor=white](-.3,-2.8)(.5,2.5)
\psframe[linecolor=darkgray,fillstyle=solid,fillcolor=darkgray](-1,.2)(-.5,2.5)
\psframe[linecolor=darkgray,fillstyle=solid,fillcolor=darkgray](-1,-.2)(-.5,-2.8)
\psframe[linecolor=lightgray,fillstyle=solid,fillcolor=lightgray](0,.2)(-.5,2.5)
\psframe[linecolor=lightgray,fillstyle=solid,fillcolor=lightgray](0,-.2)(-.5,-2.8)
\psline[linewidth=1pt](.5,2.5)(.5,1.2) \psline[linewidth=1pt](.5,.8)(.5,.2)
\psline[linewidth=1pt](.5,-.2)(.5,-2.8) \psline[linewidth=1pt](-1,2.5)(-1,2.2)
\psline[linewidth=1pt](-1,1.8)(-1,1.2) \psline[linewidth=1pt](-1,.8)(-1,.2)
\psline[linewidth=1pt](-1,-.2)(-1,-.8) \psline[linewidth=1pt](-1,-1.2)(-1,-1.8)
\psline[linewidth=1pt](-1,-2.2)(-1,-2.8)
\psarc[linewidth=1pt](.5,.5){.3}{-90}{90}
\psarc[linewidth=1pt](.5,.5){.7}{-90}{90}
\psline[linewidth=1pt](-1,.2)(.5,.2) \psline[linewidth=1pt](-1,-.2)(.5,-.2)
\psarc[linewidth=1pt](-1,.5){1.3}{90}{270}
\psarc[linewidth=1pt](-1,.5){1.7}{90}{270}
\psarc[linewidth=1pt](-1,-.5){1.3}{90}{158}
\psarc[linewidth=1pt](-1,-.5){1.3}{185}{270}
\psarc[linewidth=1pt](-1,-.5){1.7}{90}{140}
\psarc[linewidth=1pt](-1,-.5){1.7}{163}{270}
\end{pspicture}
$$
\caption{Changing the sign of a twisted band by adding two
flat annuli.} \label{3annuli}
\end{figure}

First, a flat banded surface can be constructed from a closed braid representative of a link $L$.

\begin{thm}
Let $L$ be a closed $n$-braid with a braid
word $\sigma_{1}$$\sigma_{2}$$
\ldots $$\sigma_{n-2}$ $\sigma_{n-1} W$ where the length of
$W$ is $m$ and $W$ has $s$ positive letters, then there exists a
flat banded surface $\F$ with $m + 2s$ bands such that
$\partial \F$ is isotopic to $L$, $i.e.,$ $FB(L)\le m+2s$.
\label{flatstringthm1}
\end{thm}
\begin{proof}
For a link represented by a
closed braid, Theorem~\ref{stringthm1} shows that a banded surface obtained from an $m$-string link
with each framing on the strings is either $0$ or $1$ and that
each string corresponding to a positive letter $\sigma_i$ has a framing $1$.
To replace
this band of framing $1$ by bands of framing $0$,
two extra flat bands can be added as illustrated in the second figure of Figure~\ref{3annuli}.
Once all the framings on the bands are $0$, a flat banded surface results.
The addition of $2s$ extra bands for each positive letter leads the total number of bands
in the flat banded surface to be $m+2s$. It completes the proof.
\end{proof}

The following example shows the sharp inequality in Theorem~\ref{flatstringthm1}.
A flat $4$-banded surface of the trefoil knot is illustrated in Figure~\ref{trefoil4band}.

\begin{figure}
$$
\begin{pspicture}[shift=-1.2](-.7,-1.2)(4.2,1.2)
\psarc[doubleline=true](2.5,0){1}{-5}{185}
\psarc[doubleline=true](2,0){1}{-5}{185}
\psarc[doubleline=true](1.5,0){1}{-5}{185}
\psarc[doubleline=true](1,0){1}{-5}{185}
\psframe[linecolor=lightgray,fillstyle=solid,fillcolor=lightgray](-.5,-1)(4,0)
\psline(-.03,0)(-.5,0)(-.5,-1)(4,-1)(4,0)(3.53,0)
\psline(.03,0)(.47,0) \psline(.53,0)(.97,0) \psline(1.03,0)(1.47,0)
\psline(1.53,0)(1.97,0) \psline(2.03,0)(2.47,0) \psline(2.53,0)(2.97,0)
\psline(3.03,0)(3.47,0)
\rput(1.75,-.5){{$\mathcal{D}$}}
\end{pspicture}
$$
\caption{A flat $4$-banded surface of the trefoil knot.} \label{trefoil4band}
\end{figure}
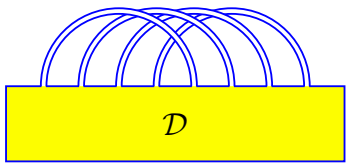

\begin{exa} \label{flattrefoil}
The flat band index of the trefoil knot is $4$.
\begin{proof}
Example~\ref{2bandnot2flat} shows that the flat band index of the trefoil knot is greater than $2$.
Since the trefoil knot is a $2$ string closed braid $\overline{(\sigma_1)^{-3}}$, we rewrite $(\sigma_1)^{-3}$ by a braid
word $\sigma_{1}$ $(\sigma_1)^{-4}$ to apply Theorem~\ref{flatstringthm1}. The braid word $W=(\sigma_1)^{-4}$ is of the length $4$ and
has no positive letters. By Theorem~\ref{flatstringthm1}, the flat band index of the trefoil knot is $\le 4$.
However, the number of components of the boundary of an $n$-banded surface is always congruent to $n+1$ modulo $2$ and
the flat band index of the trefoil knot can not be $3$. Therefore, it is $4$.
\end{proof}
\end{exa}

The upper bound of the flat band index of a link $L$ can be found from its canonical Seifert surfaces.
Obtaining a flat banded surface requires a careful choice of a disk $\D$. The spanning tree of the
induced graph of a closed braid is a path. Therefore, there is
no ambiguity about the choice of a spanning tree for a closed braid.
This is not necessarily so for a general canonical Seifert surface.
However for the choice  of any spanning tree, an alternative label on the tree
with respect to its depth from an arbitrary prefixed root satisfies the desired property,
as explained in the following lemma.

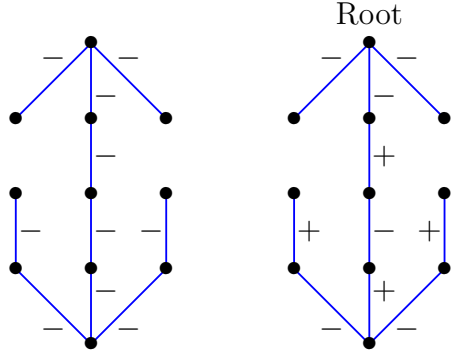
\begin{figure}
$$
\begin{pspicture}[shift=-2](-1.1,-2)(1.1,2.5)
\psline(0,2)(0,-2)(1,-1)(1,0) \psline(1,1)(0,2)(-1,1) \psline(-1,0)(-1,-1)(0,-2)
\rput(0,0){$\bullet$}
\rput(0,1){$\bullet$}
\rput(0,2){$\bullet$}
\rput(0,-1){$\bullet$}
\rput(0,-2){$\bullet$}
\rput(1,1){$\bullet$}
\rput(1,0){$\bullet$}
\rput(1,-1){$\bullet$}
\rput(-1,1){$\bullet$}
\rput(-1,0){$\bullet$}
\rput(-1,-1){$\bullet$}
\rput(.5,1.8){\rnode{a}{$-$}}
\rput(.5,-1.8){\rnode{a}{$-$}}
\rput(-.5,1.8){\rnode{a}{$-$}}
\rput(-.5,-1.8){\rnode{a}{$-$}}
\rput(.8,-.5){\rnode{a}{$-$}}
\rput(-.8,-.5){\rnode{a}{$-$}}
\rput(.2,-.5){\rnode{a}{$-$}}
\rput(.2,.5){\rnode{a}{$-$}}
\rput(.2,1.3){\rnode{a}{$-$}}
\rput(.2,-1.3){\rnode{a}{$-$}}
\end{pspicture}
\hskip 1.5cm
\begin{pspicture}[shift=-2](-1.1,-2)(1.1,2.5)
\psline(0,2)(0,-2)(1,-1)(1,0) \psline(1,1)(0,2)(-1,1) \psline(-1,0)(-1,-1)(0,-2)
\rput(0,0){$\bullet$}
\rput(0,1){$\bullet$}
\rput(0,2){$\bullet$}
\rput(0,-1){$\bullet$}
\rput(0,-2){$\bullet$}
\rput(1,1){$\bullet$}
\rput(1,0){$\bullet$}
\rput(1,-1){$\bullet$}
\rput(-1,1){$\bullet$}
\rput(-1,0){$\bullet$}
\rput(-1,-1){$\bullet$}
\rput(.5,1.8){\rnode{a}{$-$}}
\rput(.5,-1.8){\rnode{a}{$-$}}
\rput(-.5,1.8){\rnode{a}{$-$}}
\rput(-.5,-1.8){\rnode{a}{$-$}}
\rput(.8,-.5){\rnode{a}{$+$}}
\rput(-.8,-.5){\rnode{a}{$+$}}
\rput(.2,-.5){\rnode{a}{$-$}}
\rput(.2,.5){\rnode{a}{$+$}}
\rput(.2,1.3){\rnode{a}{$-$}}
\rput(.2,-1.3){\rnode{a}{$+$}}
\rput(0,2.4){\rnode{a}{$\mathrm{Root}$}}
\end{pspicture} $$
\caption{A spanning tree and its an alternating sign on the spanning tree with a root for
the induced graph $G(L(4,4,4))$ in Figure~\ref{stringgraph}.} \label{stringgraphspantree}
\end{figure}

\begin{lem}
Let $G$ be the induced graph of a Seifert surface. Let $T$ be a
spanning tree with alternating signing with respect to its depth
as depicted in Figure~\ref{stringgraphspantree}.
Let $e$ be an edge in $E(G)-T$, the sign sum of the
simple path in $T$ which joins the end points of $e$ is either $1$ or $-1$.
\label{flatstringlem1}
\end{lem}
\begin{proof}
For any two points in a tree, there exists a unique simple path between them.
Let $e$ be an edge in $E(G)-T$. Let $P$ be the unique simple path in $T$
which joins the endpoints of $e$. Because of the orientability of the Seifert surface, the length of this simple path must be odd and
the sum of the signs on the simple path must also be odd.
Furthermore, if the path $P$ does not pass the root, then the sum of the odd number of the alternating signs must be either $1$ or $-1$
depending on the parity of the length of $P$.
If $P$ passes the root, the definitions of the tree and the alternating sign show that
it is a union of two paths of alternating signs: one which starts from the root and the other which ends at the root. However,
one has odd length and the other has even length, possibly zero.
Therefore, the sign sum of any path $P$ joining the end points of an edge $e$ in
$E(G)-T$ is either $1$ or $-1$.
\end{proof}

Lemma~\ref{flatstringlem1} leads to the following minimum.
If the sign of an edge $e$ in $T$ does not coincide with the
sign of the edge in the alternating sign, then
the link must be isotoped by a type II Reidemeister move as shown in the left side of
Figure~\ref{stringgraphspandisc}. Since the signs of all the edges in the spanning tree $T$ can be reversed,
the total number of type II Reidemeister moves in the
process is less than or equal to $\left\lceil
\dfrac{s(S)-1}{2}\right\rceil$. Let $\beta$ be the total
number of type II Reidemeister moves in the process described
above. Set $\D$ as the disc corresponding to the spanning tree
$T$ as depicted in the right side of Figure~\ref{stringgraphspandisc}.
For each edge $e$ in $E(G)-T$, if the sign of the edge is
different from the sign sum of the edges in the path $P$ which joins
the endpoints of $e$, then the framing of the band presented by the edge $e$ is zero
because the linking number of $\alpha^+$ and $\alpha$ which is a union of the line representing $e$ and
the curve in $\D$ corresponds to the path $P$ in the spanning tree $T$.
Otherwise, three flat bands must be added to make the half twisted
band presented by the edge $e$.
Although two new edges are introduced by the type II Reidemeister move,
the sign of one edge is
different from the sign sum of the edges in the path $P$ which joins
the endpoints of $e$ and the sign of the other edge coincides. Therefore, there will be a total $4\beta$ bands
arising from
the type II Reidemeister moves.
Let $\gamma$ be the total
number of edges in $E(G)-T$ whose signs and
sign sums of the edges in the paths which join the endpoints of the
edges are the same and contribute $2\gamma$ extra bands to obtain a flat banded surface.
Since there are $c(S)-s(S)+1$ edges in $E(G)-T$, summarizing these facts, leads to the following theorem.

\begin{thm}
Let $\F$ be a canonical Seifert surface of a link $L$ with $s(S)$
Seifert circles and $c(S)$ half twisted bands. Let $G$ be the
induced graph of a Seifert surface $\F$. Let $T$ be a spanning tree
with alternating signing with respect to the depth of the tree.
Let $\beta$ and $\gamma$ be as described above. Then the
flat band index of $L$ is bounded by $c(S)-s(S)+ 1+  4\beta +2\gamma $, $i. e.$,
$$FB(L)\le c(S)-s(S)+ 1+  4\beta +2\gamma. $$ \label{flatstringthm3}
\end{thm}

Theorem~\ref{flatstringthm3} is directly applied to find an upper bound for the flat band index of the pretzel link $L(4,4,4)$
in the following example.

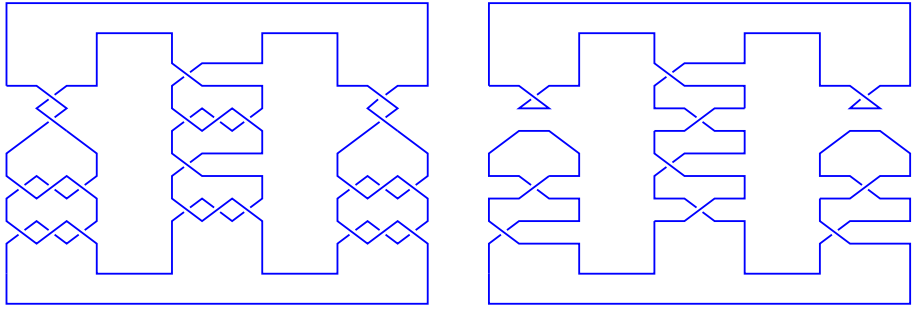
\begin{figure}
$$
\begin{pspicture}[shift=-2](-3,-2.1)(3,2.1)
\psline(-2.8,.9)(-2.8,2)(2.8,2)(2.8,.9)(2.4,.9)(2.24,.78)
\psline(2.14,.72)(2,.6)(2.4,.3)
\psline(2.24,.48)(2.4,.6)(2,.9)(1.6,.9)(1.6,1.6)(.6,1.6)(.6,1.2)(-.2,1.2)(-.36,1.08)
\psline(-.44,1.02)(-.6,.9)(-.6,.6)(-.2,.3)(.2,.6)(.6,.3)(.6,0)(-.2,0)(-.36,-.12)
\psline(-.44,-.18)(-.6,-.3)(-.6,-.6)(-.2,-.9)(.2,-.6)(.6,-.9)(.6,-1.6)(1.6,-1.6)(1.6,-1.2)(1.76,-1.08)
\psline(1.84,-1.02)(2,-.9)(2.16,-1.02) \psline(2.24,-1.08)(2.4,-1.2)(2.56,-1.08)
\psline(2.64,-1.02)(2.8,-.9)(2.8,-.6)(2.4,-.3)(2,-.6)(1.6,-.3)(1.6,0)(2,.3)(2.16,.42)
\psline(2.4,.3)(2.8,0)(2.8,-.3)(2.64,-.42)
\psline(2.56,-.48)(2.4,-.6)(2.24,-.48) \psline(2.16,-.42)(2,-.3)(1.84,-.42)
\psline(1.76,-.48)(1.6,-.6)(1.6,-.9)(2,-1.2)(2.4,-.9)(2.8,-1.2)(2.8,-1.6)(2.8,-2)(-2.8,-2)(-2.8,-1.6)
\psline(-2.8,-1.6)(-2.8,-1.2)(-2.64,-1.08)
\psline(-2.56,-1.02)(-2.4,-.9)(-2.24,-1.02) \psline(-2.16,-1.08)(-2,-1.2)(-1.84,-1.08)
\psline(-1.76,-1.02)(-1.6,-.9)(-1.6,-.6)(-2,-.3)(-2.4,-.6)(-2.8,-.3)(-2.8,0)(-2.4,.3)(-2.24,.42)
\psline(-2,.3)(-1.6,0)(-1.6,-.3)(-1.76,-.42)
\psline(-1.84,-.48)(-2,-.6)(-2.16,-.48) \psline(-2.24,-.42)(-2.4,-.3)(-2.56,-.42)
\psline(-2.64,-.48)(-2.8,-.6)(-2.8,-.9)(-2.4,-1.2)(-2,-.9)(-1.6,-1.2)(-1.6,-1.6)(-.6,-1.6)(-.6,-.9)(-.44,-.78)
\psline(-.36,-.72)(-.2,-.6)(-.04,-.72) \psline(.04,-.78)(.2,-.9)(.36,-.78)
\psline(.44,-.72)(.6,-.6)(.6,-.3)(-.2,-.3)(-.6,0)(-.6,.3)(-.44,.42)
\psline(-.36,.48)(-.2,.6)(-.04,.48) \psline(.04,.42)(.2,.3)(.36,.42)
\psline(.44,.48)(.6,.6)(.6,.9)(-.2,.9)(-.6,1.2)(-.6,1.6)(-1.6,1.6)(-1.6,.9)(-2,.9)(-2.16,.78)
\psline(-2.24,.72)(-2.4,.6)(-2,.3) \psline(-2.16,.48)(-2,.6)(-2.4,.9)(-2.8,.9)
\end{pspicture} \quad \begin{pspicture}[shift=-2](-3,-2.1)(3,2.1)
\psline(-2.8,.9)(-2.8,2)(2.8,2)(2.8,.9)(2.4,.9)(2.24,.78)
\psline(2.14,.72)(2,.6)(2.4,.6)(2,.9)(1.6,.9)(1.6,1.6)(.6,1.6)(.6,1.2)(-.2,1.2)(-.36,1.08)
\psline(-.44,1.02)(-.6,.9)(-.6,.6)(-.2,.6)(-.04,.48) \psline(-.6,.3)(-.2,.3)(.2,.6)(.6,.6)
\psline(.04,.42)(.2,.3)(.6,.3)(.6,0)(-.2,0)(-.36,-.12)
\psline(-.44,-.18)(-.6,-.3)(-.6,-.6)(-.2,-.6)(-.04,-.72)
\psline(.04,-.78)(.2,-.9)(.6,-.9)(.6,-1.6)(1.6,-1.6)(1.6,-1.2)(1.76,-1.08)
\psline(1.84,-1.02)(2,-.9)(2.8,-.9)(2.8,-.6)(2.4,-.6)(2.24,-.48)
\psline(2.16,-.42)(2,-.3)(1.6,-.3)(1.6,0)(2,.3)(2.4,.3)(2.8,0)(2.8,-.3)(2.4,-.3)(2,-.6)(1.6,-.6)
\psline(1.6,-.6)(1.6,-.9)(2,-1.2)(2.8,-1.2)(2.8,-1.6)(2.8,-2)(-2.8,-2)(-2.8,-1.6)
\psline(-2.8,-1.6)(-2.8,-1.2)(-2.64,-1.08)
\psline(-2.56,-1.02)(-2.4,-.9)(-1.6,-.9)(-1.6,-.6)(-2,-.6)(-2.16,-.48)
\psline(-2.24,-.42)(-2.4,-.3)(-2.8,-.3)(-2.8,0)(-2.4,.3)(-2,.3)(-1.6,0)(-1.6,-.3)(-2,-.3)
\psline(-2,-.3)(-2.4,-.6)(-2.8,-.6)(-2.8,-.9)(-2.4,-1.2)(-1.6,-1.2)(-1.6,-1.6)(-.6,-1.6)(-.6,-.9)
\psline(-.6,-.9)(-.2,-.9)(.2,-.6)(.6,-.6)(.6,-.3)(-.2,-.3)(-.6,0)(-.6,.3)
\psline(.6,.6)(.6,.9)(-.2,.9)(-.6,1.2)(-.6,1.6)(-1.6,1.6)(-1.6,.9)(-2,.9)(-2.16,.78)
\psline(-2.24,.72)(-2.4,.6)(-2,.6)(-2.4,.9)(-2.8,.9)
\end{pspicture}
$$\caption{A modified link diagram from the alternating signs in Figure~\ref{stringgraphspantree}
and a disc corresponding to the spanning tree with
the alternating signs.} \label{stringgraphspandisc}
\end{figure}

\begin{exa} \label{flatL444}
The flat band index of the pretzel link $L(4,4,4)$ is $\le 22$.
\begin{proof}
To apply Theorem~\ref{flatstringthm3}, the diagram of $L(4,4,4)$ is first modified as illustrated in the left side of Figure~\ref{stringgraphspandisc}.
Then consider the disc $\D$ as shown in the right side of Figure~\ref{stringgraphspandisc}.
this leads to $c(S)=12, s(S)=11$, $\beta=4$ and $\gamma=2$.
By applying Theorem~\ref{flatstringthm3}, $FB(L(4,4,4) \le 12-11+1+ 4 \cdot 4 + 2 \cdot 2 =22$.
\end{proof}
\end{exa}

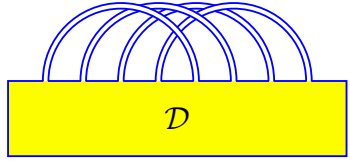
\begin{figure}
$$
\begin{pspicture}[shift=-1.2](-.7,-1.2)(4.2,1.2)
\psarc[doubleline=true](2,0){1}{-5}{185}
\psarc[doubleline=true](2.5,0){1}{-5}{185}
\psarc[doubleline=true](1.5,0){1}{-5}{185}
\psarc[doubleline=true](1,0){1}{-5}{185}
\psframe[linecolor=lightgray,fillstyle=solid,fillcolor=lightgray](-.5,-1)(4,0)
\psline(-.03,0)(-.5,0)(-.5,-1)(4,-1)(4,0)(3.53,0)
\psline(.03,0)(.47,0) \psline(.53,0)(.97,0) \psline(1.03,0)(1.47,0)
\psline(1.53,0)(1.97,0) \psline(2.03,0)(2.47,0) \psline(2.53,0)(2.97,0)
\psline(3.03,0)(3.47,0)
\rput(1.75,-.5){{$\mathcal{D}$}}
\end{pspicture}
$$
\caption{A flat $4$-banded surface of the figure eight knot.} \label{figure84band}
\end{figure}

The following example shows the inequality in Theorem~\ref{flatstringthm1} and Theorem~\ref{flatstringthm3}
fail to achieve the equality for the figure eight knot.

\begin{exa} \label{flatfigureeight}
The flat band index of the figure eight knot is $4$.
\begin{proof}
Since the figure knot is a $3$ string closed braid $\overline{(\sigma_1^{-1}\sigma_2 )^2}$,
we ruse it to apply Theorem~\ref{flatstringthm1}
but we choose the disc $\D$ presented by $\sigma_{1}^{-1}\sigma_2$ instead of $\sigma_{1}\sigma_2$ in the theorem.
The braid word $W=\sigma_{1}^{-1}\sigma_2$ is of the length $2$ and
each one has the same sign of the word representing the disc, $i.e.$, $s=2$. By Theorem~\ref{flatstringthm1}, the flat band index of the trefoil knot is $\le 2+2 \cdot 2 =6$.

To apply Theorem~\ref{flatstringthm3}, we consider the standard diagram of the figure eight knot
and its canonical Seifert surface.
Then consider the disc $\D$ which is obtained from three Seifert discs attached by any pair of half twisted
bands of different signs, this leads to $c(S)=4, s(S)=3$, $\beta=0$ and $\gamma=2$.
By applying Theorem~\ref{flatstringthm3}, the flat band index of the figure eight knot is $\le 4-3+1+ 4 \cdot 0 + 2 \cdot 2 =6$.

Example~\ref{2bandnot2flat} shows that the flat band index of the figure eight knot is greater than $2$.
Since the boundary of an $n$-banded surface has at most $n+1$ components, and the number of components is always congruent to $n+1$ modulo $2$, the flat index of the figure eight knot has to be either $4$ or $6$.
However, a flat $4$-banded surface of the figure eight knot is illustrated in Figure~\ref{figure84band}.
\end{proof}
\end{exa}

\section{Relationship between band indices and genera of links}
\label{relation}

Band index can be related with one of the classical link
invariants. First recall some definitions. The \emph{genus} of
a link $L$, denoted by $g(L)$, is the minimal genus among all its Seifert
surfaces. A Seifert surface $\F$ of $L$
with the minimal genus $g(L)$ is called a \emph{minimal genus
Seifert surface} of $L$. A Seifert surface of $L$ is
\emph{canonical} if it can be obtained from a diagram of $L$ by applying
Seifert's algorithm. The minimal genus among all canonical
Seifert surfaces of $L$ is the \emph{canonical genus} of $L$,
denoted by $g_c(L)$. A Seifert surface $\F$ of $L$ is
\emph{free} if the fundamental group of the complement of $\F$,
$\pi_1(\mathbb{S}^3 - \F)$ is a free group. The minimal
genus among all the free Seifert surfaces of $L$ is the
\emph{free genus} of $L$, denoted by $g_f (L)$. Since any canonical
Seifert surface is free, the following inequalities arise.
$$g(L)\le g_f(L) \le g_c(L).$$
These inequalities can be used to obtain many useful results~\cite{Brittenham:free, crowell:genera, KK, Moriah,
Nakamura, sakuma:minimal}.
The band index first requires the following lemma.
\begin{lem}
Let $L$ be a link, let $l$ be the number of components of $L$. Then
the band index of $L$ is bounded as,
$$ 2g(L)+ l-1  \le
B(L) \le 2g_c(L) +l-1.$$ \label{genuslem1}
\end{lem}
\begin{proof}
Let $V,E$ and $F$ be the numbers of vertices, edges and faces,
respectively in a minimal canonical embedding of $G(L)$.  From
Theorem~\ref{stringthm2}, $B(L)\le
c(\F)-s(\F)+1=E-V+1=(E-V-F)+F+1=2g_c(L)-2+F+1=2g_c(L)+F-1.$ Since a band surface is
a Seifert surface of $L$, the first
inequality follows from the definition of the genus of $L$.
\end{proof}

Consequently, the following theorem arises.
\begin{thm} \label{genusthm1}
If $L$ is a link of $l$ components and its minimal genus surface of
genus $g(L)$ can be obtained by applying Seifert algorithm on a
diagram of $L$, $i. e.$, $g(L)=g_c(L)$, then $B(L)=2g(L)+ l-1$.
\end{thm}

There are many links with coincident genera and canonical
genera such as alternating links and closures of positive
braids. The band index of such links can be found by
Corollary~\ref{genusthm1}. An example is the class of pretzel links
illustrated in Figure~\ref{npretzel}. The genera of pretzel links are
known~\cite{Gabai:genera} and their
genera and canonical genera have been shown to be the same~\cite{KL:pretzel}.
Therefore, the band index for
the following pretzel knots can be found.

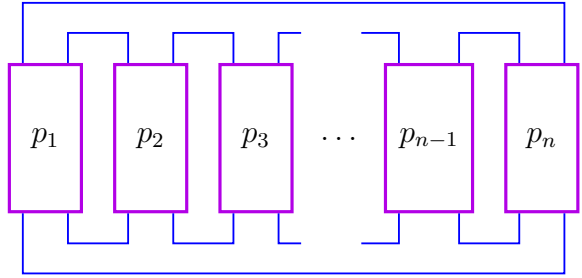
\begin{figure}
$$
\begin{pspicture}[shift=-.9](-4.5,-2)(3.1,2)
\psframe[linecolor=pup, linewidth=1.2pt](-4.2,-1)(-3.2,1) \rput(-3.7,0){$p_1$}
\psframe[linecolor=pup, linewidth=1.2pt](-2.8,-1)(-1.8,1) \rput(-2.3,0){$p_2$}
\psframe[linecolor=pup, linewidth=1.2pt](-1.4,-1)(-.4,1) \rput(-.9,0){$p_3$}
\rput(.2,0){$\ldots$} \psframe[linecolor=pup, linewidth=1.2pt](.8,-1)(2,1)
\rput(1.4,0){$p_{n-1}$} \psframe[linecolor=pup, linewidth=1.2pt](2.4,-1)(3.4,1)
\rput(2.9,0){$p_{n}$}
\psline(-4,1)(-4,1.8)(3.2,1.8)(3.2,1)
\psline(-3.4,1)(-3.4,1.4)(-2.6,1.4)(-2.6,1)
\psline(-2,1)(-2,1.4)(-1.2,1.4)(-1.2,1)
\psline(-.6,1)(-.6,1.4)(-.3,1.4) \psline(.5,1.4)(1,1.4)(1,1)
\psline(1.8,1)(1.8,1.4)(2.6,1.4)(2.6,1)
\psline(-4,-1)(-4,-1.8)(3.2,-1.8)(3.2,-1)
\psline(-3.4,-1)(-3.4,-1.4)(-2.6,-1.4)(-2.6,-1)
\psline(-2,-1)(-2,-1.4)(-1.2,-1.4)(-1.2,-1)
\psline(-.6,-1)(-.6,-1.4)(-.3,-1.4) \psline(.5,-1.4)(1,-1.4)(1,-1)
\psline(1.8,-1)(1.8,-1.4)(2.6,-1.4)(2.6,-1)
\end{pspicture}
$$
\caption{An $n$-pretzel link $L(p_1, p_2, \ldots , p_n)$}
\label{npretzel}
\end{figure}

\begin{cor}
Let $K(p_1,o_2,o_3, \ldots, o_n)$ be an $n$-pretzel knot with one
even $p_1$. Let $\alpha$ $= \sum_{i=2}^{n}$ $ sign(o_i)$ and $\beta
$$= sign(p_1)$. Suppose $|p_1|, |o_i| \ge 2$. Let
$$\delta=\sum_{i=2}^{n}(|o_{i}|-1).$$
Then the band index $B(K)$ of $K$,
$$B(K) = \left\{
\begin{array}{cl}
\delta+2  & ~~\mathrm{if}~ n~
\mathrm{is}~\mathrm{odd}~\mathrm{and}~ \alpha \neq 0, \\
 \delta & ~~\mathrm{if}~ n~
\mathrm{is}~\mathrm{even}~\mathrm{and}~
 \alpha = 0, \\
 |p_1|+\delta & ~~\mathrm{if}~ n
~\mathrm{is}~\mathrm{even}~\mathrm{and}~\alpha + \beta \neq 0, \\
|p_1|+\delta-2
  & ~~\mathrm{if}~ n ~\mathrm{is}~\mathrm{even}~\mathrm{and}~
 \alpha + \beta = 0.
\end{array} \right.
$$
 \label{genuscor1}
\end{cor}

The band indices of pretzel links of different shapes can be similarly found but they are omitted here since they are
simple consequences
of Theorem~\ref{genusthm1}.

\section*{Acknowledgments}
The authors would like to thank Tsuyoshi Kobayashi for
helpful discussions and Lee Rudolph for his comments. The referee's keen observations, which
led to Sections~\ref{string} and \ref{smallindex} taking their current form and provided the fundamentals of the proofs of  these sections' lemmas and theorems, are thankfully acknowledged. The \TeX\, macro package
PSTricks~\cite{PSTricks} was essential for typesetting the equations
and figures. The first author was supported by Basic Science Research Program through the National Research Foundation of Korea(NRF) funded by the Ministry of Education, Science and Technology(2012R1A1A2006225).

\end{document}